\documentclass[4apaper]{amsart}

\usepackage{amssymb}

\usepackage[matrix,arrow,curve]{xy}

\usepackage{amsmath, latexsym, amsfonts, pigpen, enumitem}

\usepackage[top=2cm, bottom =3cm, left=3cm, right=2cm]{geometry}

\newtheorem{theorem}{Theorem}[section]
\newtheorem{lemma}[theorem]{Lemma}

\newtheorem{corollary}[theorem]{Corollary}

\theoremstyle{definition}
\newtheorem{definition}[theorem]{Definition}
\newtheorem{example}[theorem]{Example}

\theoremstyle{remark}
\newtheorem{remark}[theorem]{Remark}

\numberwithin{equation}{section}

\newcommand{\pt}{{\rm pt}}

\newcommand{\Sing}{\operatorname{Sing}}

\newcommand{\A}{\mathbb{A}}
\renewcommand{\P}{\mathbb{P}}

\newcommand{\PP}{\mathbb{P}}

\newcommand{\id}{\operatorname{id}}

\newcommand{\Gm}{{\mathbb{G}_m}}

\newcommand{\Spec}{\operatorname{Spec}}
\newcommand{\Supp}{\operatorname{Supp}}

\newcommand{\Coker}{\operatorname{Coker}}

\newcommand{\Sm}{\mathrm{Sm}}

\newcommand{\cL}{\mathcal{L}}

\newcommand{\chark}{\operatorname{char}}
\newcommand{\pr}{\operatorname{pr}}

\newcommand{\cO}{\mathcal O}
\newcommand{\os}{\stackrel{\circ}{s}}
\newcommand{\dub}[1]{Z(I(#1)^2)}

\newcommand{\osuC}{{C^\prime}}
\newcommand{\codim}{\operatorname{codim}}
\newcommand{\G}{\mathbf G}
\newcommand{\SHd}{\mathbf{SH}^{\bullet}}

\begin{document}

\title[Framed corr. and $\pi^{n,n}_s(k)$ for perfect fields of odd characteristic]{Framed correspondences and the zeroth stable motivic homotopy group in odd characteristic}

\author{Andrei Druzhinin}
\address{\textsc{Chebyshev Laboratory, St. Petersburg State University, 14th Line V.O., 29B, Saint Petersburg 199178 Russia.}}
\email{andrei.druzh@gmail.com}
\thanks{The first author is supported by ``Native towns'', a social investment program of PJSC ``Gazprom Neft''; The author gratefully acknowledge excellent working conditions and support provided by the RCN Frontier Research Group Project no. 250399 “Motivic Hopf equations" at the University of Oslo.}

\author{Jonas Irgens Kylling}
\address{\textsc{Department of Mathematics, University of Oslo, Norway}}
\email{jonasik@math.uio.no}
\thanks{The second author was partially supported by the RCN Frontier Research Group Project no. 250399 “Motivic Hopf equations".}


\date{}

\begin{abstract}
We prove the finite descent for framed correspondences and Neshitov's moving lemma over a perfect fields.
This allows to extend the results of G.~Garkusha and I.~Panin on framed motives of algebraic varieties \cite{GP_MFrAlgVar} to finite base fields, and extend the computation of the zeroth cohomology group $H^0(\mathbb ZF(\Delta^\bullet_k,\G^{\wedge n}_m))=\mathrm{K}^\mathrm{MW}_n$, $n\geq 0$, by A.~Neshitov \cite{Nesh-FrKMW} to the case of a perfect field $k$ of odd characteristic.
\end{abstract}

\maketitle

\section{Introduction}

\subsection{Framed correspondences and Morel's theorem.} %

In the unpublished notes \cite{Voev-FrCor}
V.~Voevodsky introduced the theory of framed correspondences.
This theory grew and blossomed into the the theory of framed motives introduced and developed by G.~Garkusha and I.~Panin in \cite{GP_MFrAlgVar}, %
\cite{GP-HIVth}, \cite{AGP-FrCanc}, \cite{GNP_FrMotiveRelSphere}. 
The theory of framed motives gives an explicit fibrant resolution of motivic spectra of smooth algebraic varieties, and in particular of the motivic sphere spectrum.
A consequence is the identification of the zeroth motivic homotopy groups $\pi_{n,n}(\mathbb S)(\pt_k)$ over an infinite perfect base field $k$ with the zeroth cohomology of the Suslin complex of the presheaf of stable linear framed correspondences. %
\begin{equation}\label{eq:piH0}\pi_{-n,-n}(\mathbb S)(\pt_k)\simeq H^0(\mathbb ZF(\Delta^\bullet_k,\G_m^{\wedge n}) )\end{equation}
In \cite{Nesh-FrKMW} A.~Neshitov computed the right hand side of \eqref{eq:piH0} to be Milnor-Witt $K$-theory when the base field has characteristic zero
$$H^0(\mathbb ZF(\Delta^\bullet_k,\G_m^{\wedge n}) )\simeq \mathrm{K}_n^\mathrm{MW}(k), n\geq 0.$$
This recovers a remarkable theorem of F.~Morel \cite[Theorem 5.40]{M-A1Top} for fields of characteristic 0.

Our work extends the results of \cite{GP_MFrAlgVar} %
to finite fields, %
and extend Neshitov's computation \cite{Nesh-FrKMW} to perfect fields $k$ of odd characteristic.
This recovers Morel's theorem for perfect fields of odd characteristic.

The assumptions on the base field in this paper are as follows:
\begin{itemize}
\item In section \ref{sect:Finite Descent} the base field can be arbitrary though the interesting case is only finite fields; %
\item In section \ref{sect:StrHIFrFinF} the base field is assumed to be perfect; %
\item In section \ref{sect:KMWHoFr} the base field is assumed to be perfect of odd characteristic.
\item In section \ref{sect:MoovLm} the base field is assumed to be perfect.
\end{itemize}

\subsection{Additional ingredients and modifications}
The present text is written as a complement to the above mentioned papers of 
G.~Garkusha, I.~Panin, and A.~Ananyevskiy and A.~Neshitov. 
We only give new proofs of the statements in \cite{GP_MFrAlgVar} and \cite{Nesh-FrKMW} which require the assumptions on the base field being infinite or of characteristic 0.
Here is a list of the places that require the stronger restrictive assumptions with respect to our ones, and the list of modifications and additional arguments we use to improve the result.

1. ({\bf finite descent})
The reason of the assumption in \cite{GP_MFrAlgVar} on the base field to be infinite
are some geometrical constructions of framed correspondences and homotopies used in \cite{GP-HIVth}, namely these constructions are needed for the injectivity and excision isomorphism theorems for stable linear framed presheaves.
This leads to the restrictive assumption in the formations of 
strictly homotopy invariance theorem \cite{GP-HIVth} and cancellation theorem \cite{AGP-FrCanc} for stable linear framed presheaves, and consequently in the main results of the theory.
%
In Section \ref{sect:Finite Descent} we prove a variant of a descent for framed correspondences with respect to a set of coprime extensions.
With the descent theorem we prove the properties required by \cite{GP-HIVth}, \cite{AGP-FrCanc} and \cite{GP_MFrAlgVar} for presheaves over finite fields.

2. In Neshitov's work the assumptions stronger then the perfect fields of characteristic different from two are required because of the following.

2.1. ({\bf Steinberg relation})
Neshitov's arguments provides the homomorphism $$\Psi_*\colon \mathrm{K}^\mathrm{MW}_*(k)\to  H^0(\mathbb ZF(\Delta^\bullet_k,\G^{\wedge *}) )$$
for any field $k$, $\chark k\neq 2, 3$.
The assumption $\chark k\neq 3$ is because of the proof of the Steinberg relation in $H^0(\mathbb ZF(\Delta^\bullet_k,\G^{\wedge 2}) )$ \cite[Lemma 8.9]{Nesh-FrKMW} 
that uses a certain curve of degree 3 and traces with respect to extensions of degree 3, which play an important part of the proof.
It is the same curve as in the proof of the Steinberg relation in motivic cohomology \cite{VMW}, but 
in the case of motivic cohomology this did not lead to the restrictive assumption since derivative of the polynomial defining the curve
has no effect for $Cor$-correspondences, and it has for framed ones.\par
In the present text we replace this argument by the reference to the original geometrical proof of the Steinberg relation in $\pi^{2,2}$ by Po Hu and Igor Kritz \cite{PHandIK-Steinberg} or the alternative one \cite{Powell-Steinberg}. 
Let us refer also to \cite{DKMWhom}, where the homomorphism $\Psi_*$ is constructed for an arbitrary base scheme.

2.2. ({\bf moving lemmas})
The assumption on the base field to be of a characteristic zero and to be infinite 
is needed in the moving lemmas \cite[Lemma 4.11, Lemma 5.4]{Nesh-FrKMW},
which 
are essential ingredients providing the surjectivity
$$\Psi_*\colon \mathrm{K}^\mathrm{MW}_*(k)\twoheadrightarrow  H^0(\mathbb ZF(\Delta^\bullet_k,\G^{\wedge *}) ).$$
\cite[Lemma 4.11]{Nesh-FrKMW} allows to move an element in $\mathbb ZF_*(\pt_k,Y)$, with $Y\subset \A^n_k$ open, 
to a correspondence with the (non reduced) support being s
a set of points with the separable residue fields;
moving lemma \cite[Lemma 5.4]{Nesh-FrKMW}.
moves an element in $\mathbb ZF_*(\pt_k,\pt_k)$ 
to a correspondence with the (non reduced) support a disjoint union of rational points.

\cite[Lemma 4.11]{Nesh-FrKMW} is the main result of section 4 in \cite{Nesh-FrKMW}. The proof of lemma 
assumes that the base field is of characteristic zero because of the use of the generic smoothness theorem \cite[III, Corollary 10.7]{Hartshorne-AlG} (in \cite[Lemma 4.2, Lemma 4.6]{Nesh-FrKMW}).
Also it is needed that the field is infinite because of
the generic position lemma about hypersurfaces in projective space \cite[Lemma 4.1]{Nesh-FrKMW}.

In the article we prove \cite[Lemma 4.11]{Nesh-FrKMW} over a perfect fields of an arbitrary characteristic.
The strategy of the proof is completely different to the original one.
For a given framed correspondence the moving process consists of two parts:

Firstly, we 
move the framing functions to generic position, see Lemma \ref{prelm:GSO},
but do not modify the support of the correspondence, nor the framing functions on the first order thickening of the support.
Next, in Lemma \ref{lm:iOScormove} we change the framing functions $\phi_i$ from the $n$th to $1$st function to
obtain a so called $(i)$-\emph{simple} linear framed correspondence.
Here $(i)$-\emph{simpleness} is a "continuous version" of the notion of \emph{simpleness} of framed correspondences, see Definition \ref{def:iSmooth}. In particular a $(1)$-\emph{simple} framed correspondence is \emph{simple}.

\cite[Lemma 5.4]{Nesh-FrKMW} requires the infiniteness of the base field, since the proof
uses the existence of a separable monic polynomial in one variable with rational roots of arbitrary degree.
As shown in lemma \ref{lm:standFrptpt} the original proof of \cite[Lemma 5.4]{Nesh-FrKMW} can easily be modified to cover the case of an arbitrary field. Alternatively, the assumption can be avoided by use of the finite descent theorem of Section \ref{sect:Finite Descent}.

\subsection{Characteristic two.}
The main reason of the assumption the base field is of odd characteristic in the present proof is that 
the construction of the left inverse to  the homomorphism $\Psi_*$ uses 
the theory of Chow-Witt cohomology
introduced by Barge-Morel \cite{BM-KMW} and developed by Fasel \cite{Fasel-ChWittRing}, \cite{Fasel-GroupsdeCW}.
Actually the main problem are pushforwards for the complexes $C(X, G^n, L)_Z$, see \cite{Fasel-GroupsdeCW}.

If we replace the pushforwards of the complexes $C(X,\G^{\wedge n},L)_Z$ by the pushforwards of the Rost-Schmidt complexes of \cite[Chapter 5, Lemma 4.18]{M-A1Top} then the argument above would give a proof in arbitrary characteristic. However this would not be independent of Morel's proof. 

With out using of mentioned pushforwards 
but with using of the traces for Milnor-Witt K-theory along finite separable field extensions 
the present proof
would give a surjective homomorphism $\mathrm{K}^\mathrm{MW}_n(k)\twoheadrightarrow H^0(\mathbb ZF(\Delta^\bullet_k,\G_m^{\wedge n}) )$.

\subsection{Other work on framed motives over finite fields}
\label{subsect:OtherProofs}
Similar results on framed motives over finite fields (presented in Section \ref{sect:StrHIFrFinF}) were simultaneously and independently obtained by other authors:
In \cite[Appendix B]{ElHoKhSoYa-MotDeloop} similar results are obtained %
in terms of the conservativity property of the scalar extension functors. The basic construction on the level of correspondences is close to ours.
Also, there is an alternative construction of the descent map due to Alexey Tsybyshev based on the homomorphism $GW(k)\to \overline{\mathbb ZF}(\pt_k,\pt_k)$. %

The results on framed motives over finite fields in this paper was submitted to the Arxiv e-Print archive in a preliminary form as the Appendix of %
\cite{Appendix-FD-MotPairs}.
This was to announce the results in the first part of this work while the final parts were still work in progress. %

\subsection{Acknowledgment} 
The authors are grateful  to Ivan Panin for discussions about framed motives over finite fields and to Alexander Neshitov for discussions about extending his work \cite{Nesh-FrKMW} to fields of positive characteristic. %
The first author thanks the university of Oslo for the hospitality, where the essential part of this work was done. 

\subsection{Notation and conventions}
Throughout the text we work with explicit framed correspondences considered as classes in the category of linear framed correspondences $\mathbb ZF_*$ and framed correspondences of pairs. 
We refer the reader to \cite[Definition 2.1, Definition 8.4]{GP_MFrAlgVar}, and \cite[Definition 3.2, Definition 3.5]{GP-HIVth} 
for the definitions of categories $\mathbb ZF_*$, $\overline{\mathbb ZF}_*$, $\overline{\mathbb ZF}^{pr}_*$, and 
$\overline{\overline{\mathbb ZF}}_*$.

We need to extent the categories to the essentially smooth schemes, i.e. schemes that are localisations of smooth schemes at a point.
We use the following definition. Note that precisely in the text we never work with the correspondences with the target being except the correspondences defined by regular morphisms of essentially smooth schemes.
\begin{definition}
An essentially smooth scheme $Y$ is a scheme such that there is a sequence of open embeddings $Y_i\to Y_{i+1}$, $i\in \mathbb Z$, $Y=\varprojlim_i Y_i$. Denote by $EssSm_k$ the category of essentially smooth schemes.

For an essentially smooth scheme $Y = \varinjlim_i Y_i$ and a scheme $X$ define $\mathbb ZF_*(X,Y)=\varinjlim_i \mathbb ZF_*(X,Y_i)$.
For essentially smooth schemes $X= \varinjlim_i X_j$, $Y = \varinjlim_i Y_i$ define 
$\mathbb ZF_*(X,Y)=\varprojlim_j\varinjlim_i \mathbb ZF_*(X_j,Y_i)$.

Denote by $Sm^{pair}_k$ the full subcategory in the category of arrows in $Sm_k$ spanned by the pairs $(X,U)$ of a smooth scheme $X$ and open subscheme. 
An open pair of essentially smooth schemes $(Y,V)$ is a morphism of essentially smooth schemes $V\to Y$ that is limit of a sequence open embeddings $V_i\to Y_i$ with respect to open embeddings of pairs $(Y_i, V_i)\to (Y_{i+1},V_{i+1})$.


Define $\mathbb ZF_*((X,U),(Y,V))=\varprojlim_j\varinjlim_i \mathbb ZF_*((X_j,U_j),(Y_i,V_i))$ for pairs
$(X,U)=\varprojlim_j (X_j,U_j)$ $(Y,V)=\varinjlim_i (Y_i,V_i)$, $j,i\in \mathbb Z$.
\end{definition}

\begin{example}
Any morphism of essentially smooth schemes in $EssSm_k$ or a pair of essentially smooth schemes in the sense of above definition defines a morphism in the categories $\mathbb ZF_*$ and $\mathbb ZF_*^{pair}$. As noted above this is only one example of correspondences with the target being an essentially smooth scheme we use in the text. 
\end{example}

In addition to this list we use the following.

\begin{definition}
%

Denote by $\mathbb ZF_*^{pair}$ the factor-category of $\mathbb ZF_*^{pr}$ obtained by annihilating of the ideal generated by identity morphisms $id_{(X,X)}$ of pairs $(X,X)$, $X\in Sm_k$ ($X\in EssSm_k$).

Denote by $\overline{\mathbb ZF}_*^{pair}$ the factor-category of $\mathbb ZF_*^{pair}$ 
obtained by annihilating of the ideal generated by endomorphisms $[i_0\circ pr]-[id_{(X,U)\times\A^1}]$ of the objects $(X,U)\times\A^1$, $X\in Sm_k$, $U\subset X$ open, $i_0\colon (X,U)\to (X,U)\times\A^1$ is the zero section, $pr\colon (X,U)\times\A^1\to (X,U)$ is the projection. For a  morphisms of pairs $a\in \mathbb ZF_*^{pair}((X,U),(Y,V))$ we denote the class of $a$ in $\overline{\mathbb ZF}_*^{pair}$ by $[a]$.

For a pair of morphisms $a,b\in \mathbb ZF_*(X,Y)$ (or $a,b\in \mathbb ZF_*^{pair}((X,U),(Y,V))$) we write $a\stackrel{\A^1}{\sim} b$ iff $[a]=[b]$ in $\overline{\mathbb ZF}_*(X,Y)$ (or $\mathbb ZF_*^{pair}(X,Y)$).

Finally, denote 
${\mathbb ZF}(-, Y) = \varinjlim_i {\mathbb ZF}_n(-, Y)$, $\overline{\mathbb ZF}(-,Y) = \varinjlim_i \overline{\mathbb ZF}_n(-,Y)$, where the limits are with respect to morphisms $\sigma\colon {\mathbb ZF}_n(-, Y)\to \mathbb ZF_{n+1}(-, Y)$.
Similarly ${\mathbb ZF}(-, (Y,V)) = \varinjlim_i {\mathbb ZF}_n(-,(Y,V))$, $\overline{\mathbb ZF}(-, (Y,V)) = \varinjlim_i \overline{\mathbb ZF}_n(-, (Y,V))$.
\end{definition}
\begin{remark}
According to the above definition 
$\mathbb ZF_*^{pair}$ precisely is the category with objects being pairs $(X,U)$, $X\in \Sm_k$, $U\subset X$ open, and 
$\mathbb ZF_*^{pair}((X,U),(Y,V))$ is the homology group in the middle term of the complex
$$\mathbb ZF_*(X,V)\to \mathbb ZF_*(X,Y)\oplus \mathbb ZF_*(U,V)\to \mathbb ZF_*(U,Y).$$
\end{remark}
\begin{remark}
For two pairs $(X,U)$, $(Y,V)$ the group of morphisms $\overline{\mathbb ZF}_*^{pair}((X,U),(Y,V))$ is equal to 
$$\Coker( \mathbb ZF_*^{pair}((X,U)\times \A^1,(Y,V))\xrightarrow{i_0-i_1} \mathbb ZF_*^{pair}((X,U),(Y,V)),$$
where $i_0,i_1\colon (X,U)\to (X,U)\times \A^1$ are zero and unit sections.

The category $\overline{\mathbb ZF}_*^{pair}$
is not equal to the category $\overline{\mathbb ZF}_*^{pr}$ in  \cite{GP-HIVth}, 
but it is equal to the category $\overline{\overline{\mathbb ZF}}$. 
The only difference is the notation for morphisms; so in \cite{GP-HIVth} morphisms in $\overline{\overline{\mathbb ZF}}$ are denoted  as $[[a]]$, 
but morphisms in $\overline{\mathbb ZF}_*^{pair}$ we denote by $[a]$.
\end{remark}

\begin{remark}
Finally let us note that in the definitions of framed correspondences we usually mean implicitly the inverse image of the regular functions on $\A^n\times X$ to the neighbourhood $\mathcal V$.
So we mean $(\mathcal V, Z, v^*(\phi_1), \dots v^*(\phi_n), g)$ writing $(\mathcal V,Z,\phi_1,\dots \phi_n,g)\in Fr_n(X,Y)$,
where $Z\subset \A^n_X$, $v\colon \mathcal V\to \A^n_X$ is an etale neighbourhood of $Z$, $\phi_i\in k[\mathcal \A^n_X]$, $g\colon \mathcal V\to Y$. 
\end{remark}

\section{Finite descent}\label{sect:Finite Descent} %

In this section we prove that linear framed correspondences satisfy a descent property up to $\A^1$-homotopy with respect to a set of a finite field extensions of co-prime degrees, which we call the finite descent.
The main results are corollaries \ref{cor:FinDesN} and \ref{cor:FinDesInf}.

\begin{definition}\label{def:T_K/k}
Let $K/k$, $K=k(\alpha)$ be a separable extension of finite fields. %
Denote by $$T_{K/k} = (\Spec K, \mathcal V\to \mathbb A^1_k, f,r\colon \mathcal V\to \Spec K)\in Fr_1(\pt_k, \Spec K)$$ the framed correspondence, where
$\Spec K$ is considered as a closed subscheme of $\mathbb A^1_k$ via the function $\alpha$,
$f$ is the monic irreducible polynomial of the extension $K/k$, $\Spec K=Z(f)$,
$\mathcal V$ is an open subscheme in $\mathbb A^1_K$ that is complement $\mathcal V=\mathbb A^1_K-W$ to the closed subscheme
$W\subset \mathbb A^1_K$, such that $(\Spec K)^2 = W\amalg \Delta_K$,
where $\Delta_K$ is the graph of $\alpha\colon \Spec K\to \mathbb A^1_k$, 
the \'{e}tale morphism $v\colon \mathcal V\to \mathbb A^1_k$ is given by the composition $\mathcal V \hookrightarrow \mathbb A^1_K\to \mathbb A^1_k$ and $f$ is considered as a regular function on $\mathcal V$ under the inverse image along $v$, 
finally, $r\colon \mathcal V\to \Spec K$ is given by the projection $\mathbb A^1_K\to \Spec K$.

\end{definition}
\begin{definition}
Let $\Lambda_l\in Fr_1(\pt_k,\pt_k)$ be the framed correspondence defined by the function $x^l$ on $\mathbb A^1_k$ (i.e., $( 0, \mathbb A^1_k = \mathbb A^1_k, x^l, \mathbb A^1_k \to \Spec k)$).

Let $\Lambda^\prime_l \in \mathbb ZF_1(\pt_k,\pt_k)$ be the framed correspondence given by $hm$ if $l=2m$ and $hm+\langle 1\rangle$ if $l=2m+1$,
where $h=\langle 1\rangle +\langle -1\rangle$ is the hyperbolic plane.
Here $\langle 1 \rangle\in Fr_1(\pt_k,\pt_k)$ is the framed correspondence defined by the function $x$ on $\mathbb A^1$ (i.e., $\Lambda_1$), and $\langle -1\rangle\in Fr_1(\pt_k,\pt_k)$ is the framed correspondence defined by the function $-x$ on $\mathbb A^1$.
\end{definition}

\begin{definition}
Let $c=(Z,\mathcal V,\phi,g)\in Fr_n(X,Y)$ be an explicit framed correspondence such that $\mathcal V\subset \A^n_X$ is a Zariski neighbourhood of $Z$, and the regular functions $\phi_i$, $i=1,\dots n$ on $\mathcal V$ are restrictions of globally defined regular functions on $\A^n_X$. 
Then to shorten the notations we often omit writing either the support $Z$ or $\mathcal V$, and write $c=(\mathcal V,\phi,g)$ or $c=(Z,\phi,g)$. 

If moreover, $\mathcal V = \A^n_X$, we will write just $c=(\phi,g)$; or
if $Y=\pt_k$ then we omit the canonical map $g$ and write $c=(\mathcal V,\phi)$ or $c=(Z,\phi)$.

We denote by $\langle \lambda\rangle \in Fr_1(\pt_k,\pt_k)$ the framed correspondence given by $(\A^1_k,\lambda x)$.
\end{definition}

\begin{lemma}\label{sbl:Lambda}
The classes of $\Lambda_l$ and $\Lambda^\prime_l$ in the group $\overline{\mathbb ZF}_1(\pt_k,\pt_k)$ coincide.
For any positive integers $l_1,l_2$ and any $n_1,n_2$ such that $l_1 n_1 - l_2n_2=1$ we have $[\Lambda^\prime_{l_1} \circ \Lambda^\prime_{l_2} - \Lambda^\prime_{l_2}\circ \Lambda^\prime_{n_2}]=[\langle (-1)^{l_2n_2}\rangle]\in\overline{\mathbb ZF}(\pt_k,\pt_k)$.
\end{lemma}
\begin{proof}

The first statement is \cite[remark 7.8]{Nesh-FrKMW}.
For the readers convenience we repeat the proof:
Firstly we note that in the group of linear framed correspondences $\mathbb ZFr_1(\pt_k,\pt_k)$ we have
\[
( \mathbb A^1_k,  x^l(1+x) )
=
( \mathbb A^1_k-\{0\},  x^l(1+x) )
+
( \mathbb A^1_k-\{-1\},  x^l(1+x) )
\sim
\Lambda_{l}
+ \langle (-1)^{l} \rangle,
\]
where the first equality holds by the definition of the group of linear framed correspondences $\mathbb ZF_1(\pt_k,\pt_k)$.
Then by induction it follows that $[\Lambda_l]=[\sum\limits_{i=0}^{l-1}\langle (-1)^i\rangle]=[\Lambda^\prime_l]\in \overline{\mathbb ZFr}_1(\pt_k,\pt_k)$.

The second statement is then straightforward.
\end{proof}

\begin{lemma}\label{sblm:prT}
For any separable finite extension $K/k$ we have
$pr_{K/k}\circ T_{K/k}\stackrel{\mathbb A^1}{\sim} \Lambda_{deg\, K/k}$,
where $\pr_{K/k} : \Spec K \to \Spec k$.
\end{lemma}
\begin{proof}
The claimed equivalence is provided by the homotopy defined by the framed correspondence
$(\A^1_k\times {\A^1_k},\lambda f +(1-\lambda) x^{\deg K/k},pr)\in Fr_1(\A^1_k,pt_k)$,
where
$f$ is the monic polynomial from the definition \ref{def:T_K/k},
$\lambda$ denotes the second coordinate on $\A^1_k\times\A^1_k$, which is the homotopy parameter,
and $pr\colon \A^1_k\times{\A^1_k}\to pt_k$ is the canonical projection.
\end{proof}

\begin{definition}\label{def:precatdiag}
We call by {\em a small precategory} $\Gamma=(Ob_\Gamma, Mor_\Gamma, U_\Gamma, \circ)$ a set of vertices $Ob_\Gamma$, a set of arrows $Mor_\Gamma$, 
a subset $U_\Gamma\in Mor_\Gamma\times Mor_\Gamma$, and a map $-\circ-\colon U\to Mor_\Gamma$
such that $f\circ (g\circ h)=(f\circ g)\circ h$ for any $f,g,h\in Mor_{\Gamma}$, $(f,g), (g,h), (f\circ g,h), (f,g\circ h)\in U$. 
A {\em finite precategory} is a {\em precategory} $\Gamma$ with finite sets $Ob_\Gamma$ and $Mor_\Gamma$. 

A \emph{functor of small precategories} $F\colon\Gamma_1\to\Gamma_2$ is a pair of maps
$F_{Ob}\colon Ob_{\Gamma_1}\to Ob_{\Gamma_2}$ and 
$F_{Mor}\colon Mor_{\Gamma_1}\to Mor_{\Gamma_2}$. 
We call the functors $F$ with $F_{Ob}$ and $F_{Mor}$ being injective as \emph{embedding} of (small) precategories.
The \emph{category of small precategories} is a category with objects being small precategories
and morphisms being functors of small precategories. 
\end{definition}
\begin{definition}
A $\Gamma$-diagram in the category $\mathcal F$ is a pairs of maps $f_{Ob}\colon Ob_\Gamma\to Ob_{\mathcal F}$ and $f_{Mor}\colon Mor_\Gamma\to Mor_{\mathcal F}$ such that for any arrows $\alpha_1, \alpha_2, \alpha_3\in Mor_\Gamma$, $\alpha_3=\alpha_1\circ\alpha_2$, $(\alpha_1,\alpha_2)\in U_\Gamma$, we have $f_{Mor}(\alpha_1)\circ f_{Mor}(\alpha_2)=f_{Mor}(\alpha_3)$.
\end{definition}

\begin{definition}\label{def:GammaGammaprimeGoodwrtDesctnt}
Let $\Gamma^\prime\to \Gamma$ be an embedding of a precategories  (in sense of def. \ref{def:precatdiag}).
We say that the embedding is good with respect to descent if and only if
for each pair of morphisms $\gamma_1,\gamma_2 \in \Gamma$ such that the composite $\gamma_1\circ\gamma_2$ is defined,
$\gamma_1$ or $\gamma_2$ is the image of a morphism in $\Gamma^\prime$.
\end{definition}

\begin{example}
1) Any small category $\Gamma$ is a small precategory with $U_\Gamma=Mor_{\Gamma}$.
2) Any graph $\Gamma$ can be considered as a small precategory with $U_\Gamma=\emptyset$.
3) Main examples of the embeddings of small precategories we will work with are represented by the diagrams
$$
\xymatrix{
Ob_1\ar@{-->}[d]^{f_1}\ar[dr]^{f_3}
\\
Ob_2\ar[r]^{f_2} & Ob_3
}\quad
\xymatrix{
Ob_1\ar[dr]|<<<{g_5}\ar[r]^{g_1}\ar[d]^{g_3} & Ob_2\ar[d]^{g_4}
\\
Ob_3\ar[r]^{g_2}\ar@{-->}[ru]|>>>{g_6} & Ob_4
}
$$
where the left diagram represents
the embedding $\Gamma_1\to \Gamma_2$ with 
$Ob_{\Gamma_1}=Ob_{\Gamma_2}=\{V_1,V_2,V_3\}$,
$Mor_{\Gamma_1} = \{f_1,f_2\}$, $Mor_{\Gamma_1} = \{f_1,f_2,f_3\}$, $f_2\circ f_1=f_3$,
and
the embedding $\Gamma_3\to \Gamma_4$ with 
$Ob_{\Gamma_3}=Ob_{\Gamma_4}=\{V_4,V_5,V_6,V_7\}$,
$Mor_{\Gamma_1} = \{g_1,g_2,g_3,g_4,g_5\}$, $Mor_{\Gamma_1} = \{g_1,g_2,g_3,g_4,g_5,g_6\}$, 
$g_2\circ g_3=g_5$, $g_4\circ g_1=g_5$,
$g_4\circ g_6=g_2$, $g_6\circ g_3=g_1$.
\end{example}

\begin{definition}\label{def:Weak(Hom)Diagram}
Let $\Gamma$ be a small precategory.

\emph{A weak $\Gamma$-diagram} in the category $\mathbb ZF_*(k)$  is
a map $\gamma\colon \Gamma\to \mathbb ZF_*(k)$ such that for any morphisms $\alpha_1\circ\alpha_2=\alpha_3\in Mor_{\Gamma}(a,c)$, we have $[\gamma(\alpha_1)\circ\gamma(\alpha_2)]=[\gamma(\alpha_3)]\in \mathbb ZF(\gamma(a),\gamma(c))$. 

\emph{A weak homotopy $\Gamma$-diagram} in $\mathbb ZF_*(k)$ is a 
weak $\Gamma$-diagram in $\overline{\mathbb ZF}_*(k)$.
Precisely it is a map $\gamma\colon \Gamma\to \mathbb ZF_*(k)$
such that for any morphisms $\alpha_1\circ\alpha_2=\alpha_3\in Mor_{\Gamma}(a,c)$, 
we have 
$[\gamma(\alpha_1)\circ\gamma(\alpha_2)]=[\gamma(\alpha_3)]\in \overline{\mathbb ZF}(\gamma(a),\gamma(c))=\mathbb ZF(\gamma(a),\gamma(c))/{\sim}_{\A^1}$.

Weak (homotopy) $\Gamma$-diagrams in the category of pairs $\mathbb ZF^{pair}_*(k)$ are defined similarly.
\end{definition}

\begin{definition}\label{def:WeakLP}

Let $j\colon \Gamma^\prime\to \Gamma$ be an embedding of precategories  (in sense of def. \ref{def:precatdiag}) that is good with respect to descent (in sense of def. \ref{def:GammaGammaprimeGoodwrtDesctnt}),
and let $\gamma^\prime$ 
be a $\Gamma^\prime$-diagram in the category 
$\mathbb ZF_*$ (or the category of pairs $\mathbb ZF^{pair}_*$).

We say that we have 
\emph{the weak lifting property} in the category ${\mathbb ZF}_*$
with respect to $j$ and $\gamma^\prime$ iff 
there is a weak $\Gamma$-diagram in $\mathbb ZF_*(K)$ (def. \ref{def:Weak(Hom)Diagram}) that restriction on $\Gamma^\prime$ is  
$\gamma'$. %

Similarly we define weak lifting properties in $\overline{\mathbb ZF}_*$,
and weak lifting properties in 
${\mathbb ZF}^{pair}_*$, and $\overline{\mathbb ZF}^{pair}_*$
with respect to a $\Gamma^\prime$-diagram in the category of pairs.
\end{definition}

\begin{lemma}\label{rem:represGamD}
1) For any small (finite) precategory $\Gamma$  (in sense of def. \ref{def:precatdiag}) there is a small (finite) precategory $\Gamma_s$ such that 
the category of week
$\Gamma$-diagrams in the category $\overline{\mathbb ZF}_*$ (or $\overline{\mathbb ZF}_*^{pair}$)
is equivalent to the category 
$\Gamma_s$-diagrams in the category $\mathbb ZF_*(k)$ (or $\mathbb ZF_*^{pair}$). 

2)
For any small (finite) precategory $\Gamma$  (in sense of def. \ref{def:precatdiag}) there is a small (finite) precategory $\Gamma_s$ such that 
the category of week
$\Gamma$-diagrams in the category $\mathbb ZF_*^{pair}$
is equivalent to a full subcategory in the category of weak 
$\Gamma_s$-diagrams in the category $\mathbb ZF_*(k)$, 
defined by the condition that some arrows in the (week) $\Gamma_s$-diagram are open embeddings of schemes. 

\end{lemma}
\begin{proof}
We note the following:
(1) An equality $[\Phi_1]=[\Phi_2]\in\overline{\mathbb ZF_*}(X_1,X_2)$ is represented by a morphism $\mathbb ZF_*(X_1\times\A^1,X_2)$.
(2) Any morphism in the category of pairs $(X,U)\to (Y,V)$ in the category $\mathbb ZF_*(k)$ is represented by a commutative square in the category of correspondences 
\begin{equation}\label{eq:sqPair}\xymatrix{
X\ar@{-->}[rd]\ar[r] & Y\\
U\ar[u]\ar[r] & V\ar[u]
}
\end{equation}
An equality in the category of pairs is equivalent to the existence of the diagonal in the square above.

The lemma is proven in two steps: 
1) Firstly, using observation (1)
we can replace all equalities in the diagram $\overline{\mathbb ZF}^{pair}_*(k)$ of the form $\gamma(\alpha_1)\circ\gamma(\alpha_2)=\gamma(\alpha_1\circ\alpha_2)$ by equalities in $\mathbb Z F^{pair}_*(k)$.
Each source vertex $(X,U)$ of the arrow $\alpha_2$ is replaced by the triple $(X,U)\times0\to (X,U)\times\A^1\leftarrow (X,U)\times 1$.
We denote the resulting diagram by $\gamma^\prime$.

2) Next we replace each vertex $(X,U)$ of $\gamma^\prime$ by the pair of vertices $X\leftarrow U$, we replace each arrow by a square of the form \eqref{eq:sqPair}, and add a diagonal arrow to the square for each relation of the form $\gamma^\prime(\alpha_1)\circ\gamma^\prime(\alpha_2)=\gamma^\prime(\alpha_1\circ\alpha_2)$.
\end{proof}

\begin{lemma}\label{lm:FinDescent}
Let $\Gamma^\prime\to \Gamma$ be an embedding of precategories  (in sense of def. \ref{def:precatdiag}) that is good with respect to descent (def. \ref{def:GammaGammaprimeGoodwrtDesctnt}),
and let $\gamma^\prime$ %
be a $\Gamma^\prime$-diagram in the category $\overline{\mathbb ZF}_*$ (or the category of pairs $\overline{\mathbb ZF}^{pair}_*$). %

Let $K_1$, $K_2$ be two finite field extensions of a finite field $k$,
such that $\deg K_1/k$ and $\deg K_2/k$ are relatively prime, i.e., $(\deg K_1/k, \deg K_2/k)=1$.
Suppose that there exist lift of the weak $\Gamma'$-diagram $\gamma'$ (def. \ref{def:Weak(Hom)Diagram})
to a weak $\Gamma$-diagram in $\overline{\mathbb ZF}_*(K)$ (or $\overline{\mathbb ZF}^{pair}_*(K)$) for $K = K_1, K_2$. %

Then there exists a lift of $\gamma'$ to a weak $\Gamma$-diagram in $\overline{\mathbb ZF}_*(k)$ (or $\overline{\mathbb ZF}_*^{pair}(k)$) (see def. \ref{def:Weak(Hom)Diagram}).
\end{lemma}
\begin{proof}
The question on the diagrams in the categories $\overline{\mathbb ZF}_*$, the case of $\overline{\mathbb ZF}^{pair}_*$ follows form the case of 
$\mathbb ZF_*$, the case of $\mathbb ZF^{pair}_*$ by the first point of lemma \ref{rem:represGamD}.
Now consider the case of $\mathbb ZF_*$, and the case of $\mathbb ZF^{pair}_*$ is similar. 

Let $l_1 = \deg K_1$, $l_2 = \deg K_2$.
Since $(l_1,l_2)=1$ there are integers $n_1, n_2$ such that $l_1 n_1 - l_2n_2=1$.
Up to remuneration of $K_1$ and $K_2$ we can assume that $n_1, n_2>0$.
Let $S=\Spec K_1\amalg \Spec K_2$, with inclusions $j_i : \Spec K_i \hookrightarrow S, i=1,2$.
Define a framed correspondence
\begin{equation}\label{eq:invlift}L = \langle (-1)^{l_1 n_1+1}\rangle \circ (j_1\circ T_{K_1/k}\circ \Lambda_{n_1} - j_2 \circ T_{K_2/k}\circ \Lambda_{n_2})\in \mathbb ZF_1(\pt_k, S),\end{equation}
and let $pr\colon S\to \Spec k$ be the projection.
Then Lemma \ref{sbl:Lambda} and Lemma \ref{sblm:prT} imply that \begin{multline}\label{eq:DescL}
[pr \circ L]=
[\langle (-1)^{l_1 n_1+1}\rangle\circ pr \circ (j_1\circ T_{K_1/k}\circ \Lambda_{n_1} - j_2 \circ T_{K_2/k}\circ \Lambda_{n_2})]
\stackrel{lm \ref{sblm:prT}}{=}\\
[\langle (-1)^{l_1 n_1+1}\rangle \circ (\Lambda_{l_1}\circ \Lambda_{n_1} - \Lambda_{l_2}\circ \Lambda_{n_2})]
\stackrel{lm \ref{sbl:Lambda}}{=}
[\sigma]\in \overline{ \mathbb ZF}_1(\pt_k,\pt_k).\end{multline}

Consider the base change $\gamma_S^\prime\colon \mathbb ZF_*(S)$ of the $\Gamma^\prime$-diagram $\Gamma^\prime$, where $\mathbb ZF_*(S)$ denotes the category of framed correspondences over $S$.
By assumption there is a lift of $\gamma^\prime_S$ to a $\Gamma$-diagram $\gamma_S\colon \Gamma\to \mathbb ZF_*(S)$.
Denote by $L_S$ the base change of $L$.

To define the required lift $\gamma$ for any morphism 
$\alpha\in \Gamma$ that is not a morphism in $\Gamma^\prime$ 
we put
\begin{equation}
\label{eq:gamma-alpha}
\gamma(\alpha)= 
\big( (id_Y\boxtimes pr) \circ \gamma_{S}(\alpha) \circ (X\times L) \big) \cdot \langle (-1)^{l_1n_1}\rangle,
\end{equation}
where $X$ and $Y$ are varieties (or pairs) that are the source and the target of $\gamma_S(\alpha)$, $X\times L_S$ is the base change of $L_S$ with respect to $X\to \pt_k$, and $\cdot$ is the external product of correspondences
(if $\alpha\in \Gamma^\prime$ we put $\gamma(\alpha)=\gamma^\prime(\alpha)$).
Denote $L_X = L\boxtimes id_{X}$, for $X\in Sm_k$.

Now let $\alpha,\beta\in \Gamma$ be a pair of morphisms such that the target of $\alpha$ is equal to the source of $\beta$.
Let $X$ and $Y$ be images of the source and target of $\alpha$ in $\overline{\mathbb ZF}_*(K)$,
and let $Y$ and $Z$ be images of the sources and target of $\beta$ in $\overline{\mathbb ZF}_*(K)$.
Since the embedding $\Gamma^\prime\to \Gamma$ is good with respect to descent,
for any such a pair either $\alpha\in \Gamma^\prime$, or $\beta\in \Gamma^\prime$.

Suppose $\beta\in \Gamma^\prime$, then
\begin{multline*}
[\gamma(\beta) \circ  \gamma(\alpha)]=
[\gamma^\prime(\beta) \circ  \gamma(\alpha)]=
[\gamma^\prime(\beta) \circ pr_Y\circ \gamma_S(\alpha)\circ L_X]=\\
[pr_Z\circ\gamma^\prime_S(\beta)\circ \gamma_S(\alpha)\circ L_X]=
[pr_Z\circ\gamma_S(\beta)\circ \gamma_S(\alpha)\circ L_X]=\\
[pr_Z\circ\gamma_S(\beta\circ \alpha)\circ L_X]=
[\gamma(\beta\circ \alpha)]
\end{multline*}
since we have the diagram
$$\xymatrix{
S\times X \ar[r]^{\gamma_S(\alpha)} & S\times Y \ar[r]^{\gamma_S(\beta)}^{\gamma^\prime_S(\beta)}\ar@<1ex>[d]^{pr_Y} & S\times Z\ar[d]^{pr_Z} \\
X\ar[r]^{\gamma(\alpha)}\ar[u]^{L_X} &Y\ar@<1ex>[u]^{L_Y} \ar[r]_{\gamma^\prime(\beta)}^{\gamma(\beta)} & Z.
}$$

Now suppose $\alpha\in \Gamma^\prime$, then
\begin{multline*}
[\gamma(\beta) \circ  \gamma(\alpha)]=
[\gamma(\beta) \circ  \gamma^\prime(\alpha)]=
[pr_Z\circ \gamma_S(\beta) \circ L_Y\circ \gamma^\prime(\alpha)]=\\
[pr_Z\circ \gamma_S(\beta) \circ \gamma^\prime_S(\alpha)\circ L_X]=
[pr_Z\circ \gamma_S(\beta) \circ \gamma^\prime_S(\alpha)\circ L_X]=
[pr_Z\circ \gamma_S(\beta) \circ \gamma_S(\alpha)\circ L_X]=\\
[pr_Z\circ \gamma_S(\beta\circ \alpha)\circ L_X]=
[pr_Z\circ \gamma_S(\beta\circ \alpha)\circ L_X]=
[\gamma(\beta\circ \alpha)]
\end{multline*}
since we have the diagram
$$\xymatrix{
S\times X \ar[r]^{\gamma_S(\alpha)}_{\gamma^\prime_S(\alpha)} & S\times Y \ar[r]^{\gamma_S(\beta)}\ar@<1ex>[d]^{pr_Y} & S\times Z\ar[d]^{pr_Z} \\
X\ar[r]^{\gamma(\alpha)}_{\gamma^\prime(\alpha)}\ar[u]^{L_X} &Y\ar@<1ex>[u]^{L_Y} \ar[r]^{\gamma(\beta)} & Z.
}$$
\end{proof}

\begin{corollary}\label{cor:FinDesN}
Let $\Gamma^\prime\to \Gamma$ be an embedding of precategories  (in sense of def. \ref{def:precatdiag}) that is good with respect to descent (def. \ref{def:GammaGammaprimeGoodwrtDesctnt}),
and let $\gamma^\prime$ %
be a $\Gamma^\prime$-diagram in the category $\overline{\mathbb ZF}_*(k)$ (or the category of pairs $\overline{\mathbb ZF}^{pair}_*(k)$).

Suppose that there is an integer $N$ such that for all field extensions $K/k$ of degree $\deg K/k\geq N$
there is a lift of $\gamma^\prime$
to a weak $\Gamma$-diagram (def. \ref{def:Weak(Hom)Diagram}) in the category $\overline{\mathbb ZF}_*(K)$ (or $\overline{\mathbb ZF}_*^{pair}(K)$).

Then there is such a lift of $\gamma^\prime$ 
over $k$.
\end{corollary}
\begin{proof}
Consider any two separable extensions $K_1/k$, $K_2/k$ such that 
$\deg K_1,\deg K_2>N$, $(\deg K_1,\deg K_2)=1$.
Then by assumption the required lift of the diagram $\gamma^\prime$ over $K_1$ and $K_2$,
so the claim follows from lemma \ref{lm:FinDescent}.
\end{proof}

\begin{corollary}[Tsybyshev, Panin]\label{cor:FinDesInf}
Let $\Gamma^\prime\to \Gamma$ be an embedding of finite precategories (in sense of def. \ref{def:precatdiag}) that is good with respect to descent (def. \ref{def:GammaGammaprimeGoodwrtDesctnt}),
and let $\gamma^\prime$ %
be a $\Gamma^\prime$-diagram in the category $\overline{\mathbb ZF}_*(k)$ (or the category of pairs $\overline{\mathbb ZF}^{pair}_*(k)$).

Suppose that 
for all infinite field extensions $K/k$
there is a lift of $\gamma^\prime$
to a weak $\Gamma$-diagram (def. \ref{def:Weak(Hom)Diagram}) in the category $\overline{\mathbb ZF}_*(K)$ (or $\overline{\mathbb ZF}_*^{pair}(K)$).

Then there is such a lift of $\gamma^\prime$ 
over $k$.

In other words, if the weak lifting property (in sense of def. \ref{def:WeakLP}) with respect to the diagram $\gamma^\prime$ and the embedding $\Gamma^\prime\to \Gamma$ holds in the category 
$\overline{\mathbb ZF}_*(K)$ (or $\overline{\mathbb ZF}_*^{pair}(K)$) for any infinite field extension $K/k$ then it holds in $\overline{\mathbb ZF}_*(k)$ (or $\overline{\mathbb ZF}_*^{pair}(k)$).
\end{corollary}
\begin{proof}
The claim follows from lemma \ref{lm:FinDescent} if we consider the towers of extensions of degrees $p$ and $q$ for two different prime numbers, prime to the characteristic. 

In detail, let $K^\prime_1 = \varinjlim_l k(\xi^{1/p^l})$ and $K_2^\prime = \varinjlim_l k(\xi^{1/q^l})$ 
be the infinite extensions of $k$ for a pair of different prime numbers $p,q$, $(p,\chark k)=1,\, (q,\chark k)=1$.
Then by assumption there is a lift of the $\Gamma^\prime$-diagram $\gamma^\prime$ to a weak $\Gamma$-diagram. Since the precategory $\Gamma$ is finite it follows that there is a lift of $\gamma^\prime$ to a weak $\Gamma$-diagram over a finite field extensions $K_1=k(\xi^{1/p^l})$ and $K_2=k(\xi^{1/q^l})$. Now the claim follows from lemma \ref{lm:FinDescent}.
\end{proof}

\begin{lemma}\label{lm:bigK-Aff}
Let $B\subset \A^n_k$ be a closed subscheme in affine $n$-space over some field $k$.
Then there is $N\in \mathbb Z$ such that for all field extensions $K/k$ of degree $\deg K/k>N$ there is a $K$-rational point $p\in \A^n_K - B_K$, where $B_K=B\times \Spec K$.
\end{lemma}
\begin{proof}
Let $f\in k[t_1,\dots,t_n]$ be a function, $f\big|_B=0$, $f\neq 0$.
Assume $f\vert_{\mathbb A^{n-1}_K\times P} = 0$ for any $P \in \mathbb A^1_k(K)$. This is impossible if $f$ is nonzero and $\deg K/k \gg 0$,
since the number of rational roots of a nonzero one-variable polynomial is not greater than its degree.
\end{proof}
%

%


\section{Injectivity and excision theorems for framed presheaves over a field.}

In the section we apply the result of the previous section to extend the injectivity and excision theorems for homotopy invariant linear stable framed presheaves \cite[]{} to the finite base field case. 

\begin{theorem}[injectivity on local schemes]
\label{th:InjLoc}
For a field $k$ let $X \in Sm_k$, $x \in X$ be a point, $U = \Spec(O_{X,x})$, $i\colon D\to X$ be a proper closed subset. Then there exists an integer $N$ and a morphism $r\in \mathbb ZF_N(U,X-D)$ such that
$[r]\circ [j] = [can]\circ [\sigma_U^N]$
in $\overline{\mathbb ZF}_N(U ,X)$ with $j \colon  X-D\to X$ the open inclusion and $can\colon U \to X$ the canonical morphism.
\end{theorem}
\begin{theorem}[injectivity on affine line] 
\label{th:InjA1}
For a field $k$ let $U \subset \A^1_k$ be an open subset and let $i \colon  V\to U$ be a non-empty open subset. Then there is a morphism $r\in\mathbb ZF_1(U,V)$ such that $[i]\circ [r]=[\sigma_U]\in \overline{\mathbb ZF}_1(U,U)$.
\end{theorem}
\begin{theorem}[Zariski excision on affine line]
\label{th:ZarexA1}
For a field $k$ let $U \subset \A^1_k$ be an embedding. 
Let $i\colon V \to U$ be an open inclusion with $V$ non-empty. 
Let $S \subset V$ be a closed subset. 
Then there are morphisms $r \in\mathbb ZF_1((U,U -S),(V,V -S))$ and $l \in \mathbb ZF_1((U,U -S),(V,V -S))$ such that
$$[i] \circ [r]] = [\sigma_U ] \text{ and } [i] \circ [l] = [\sigma_V ]$$
in $\overline{\mathbb ZF}^{pair}_1((U,U -S),(U,U -S))$ and $\overline{\mathbb ZF}^{pair}_1((V,V -S),(V,V -S))$ respectively.
\end{theorem}
\begin{theorem}[\'{e}tale excision]\label{th:etex}
Let $S \subset X$ and $S^\prime \subset X^\prime$ be closed subsets. Let
$$\xymatrix{
V^\prime ￼ \ar[r]\ar[d]& X^\prime \ar[d]^{\Pi}\\
V \ar[r]& X
}$$
be an elementary distinguished square with $X$ and $X^\prime$ affine $k$-smooth. Let $S = X -V$ and $S^\prime = X^\prime - V^\prime$ be closed subschemes equipped with reduced structures. Let $x \in S$ and $x^\prime \in S^\prime$ be two points such that $\Pi(x^\prime) = x$. Let $U = \Spec(\mathcal O_{X,x})$ and $U^\prime= \Spec(O_{X^\prime,x^\prime} )$. Let $\pi : U^\prime \to U$ be the morphism induced by $\Pi$.

Under the notation above there is an integer $N$ and a morphism $r\in \mathbb ZF_N((U,U -S),(X^\prime,X^\prime-S^\prime))$ such that
$$[\Pi] \circ [r] = [can] \circ [\sigma_U^N ]$$
in $\overline{\mathbb ZF}^{pair}_N((U,U -S),(X,X -S))$, where $can \colon U \to X$ is the canonical morphism.

There are an integer $N$ and a morphism $l \in \mathbb ZF_N((U,U -S),(X^\prime,X^\prime-S^\prime))$ such that
$$[l]\circ [\pi]=[can^\prime]\circ [\sigma_U^{N^\prime}]$$
in $\overline{\mathbb ZF}_N^{pair}((U^\prime,U^\prime -S^\prime),(X^\prime,X^\prime-S^\prime))$ with 
$can^\prime\colon U^\prime\to X^\prime$ the canonical morphism.
\end{theorem}

\begin{lemma}\label{lm:Diagrams}
Any theorem of the list \ref{th:InjLoc}-\ref{th:etex} 
states a weak lifting property (in sense of def .\ref{def:WeakLP})
in the categories $\overline{\mathbb ZF}_*$ (or $\overline{\mathbb ZF}^{pair}_*$)
with respect to some diagram $\gamma^\prime\colon \Gamma^\prime\to \overline{\mathbb ZF}_*$ (or $\gamma^\prime\colon \Gamma^\prime\to \overline{\mathbb ZF}^{pair}_*$)
and an embedding $\Gamma^\prime\to \Gamma$ good with respect to a descent (see def \ref{def:GammaGammaprimeGoodwrtDesctnt}) for a finite precategories $\Gamma$ and $\Gamma^\prime$  (in sense of def. \ref{def:precatdiag}).
\end{lemma}
\begin{proof}
The list of diagrams follows.
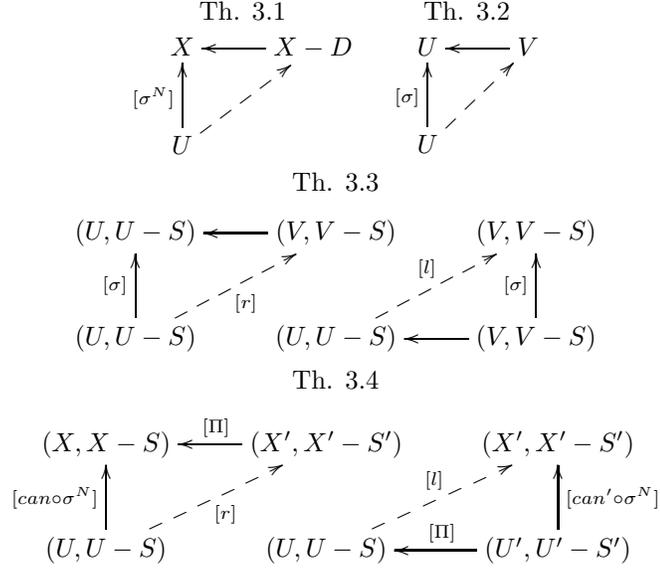
\begin{gather*}
\begin{array}{cc}
\text{Th. \ref{th:InjLoc}} & \text{Th. \ref{th:InjA1}}\\
\xymatrix{
X & X-D\ar[l] \\ U\ar[u]^{[\sigma^N]}\ar@{-->}[ru]
}&
\xymatrix{
U & V\ar[l] \\ U\ar[u]^{[\sigma]}\ar@{-->}[ru]
}\end{array}
\\
\text{Th. \ref{th:ZarexA1}}\\
\xymatrix{
(U,U -S) &(V,V -S)\ar[l] &(V,V -S)\\ (U,U -S)\ar[u]^{[\sigma]}\ar@{-->}[ru]_{[r]} & (U,U -S)\ar@{-->}[ru]^{[l]} & (V,V -S)\ar[u]^{[\sigma]}\ar[l] 
}\\
\text{Th. \ref{th:etex}}\\
\xymatrix{
(X,X -S) &(X^\prime,X^\prime-S^\prime)\ar[l]_{[\Pi]}  & (X^\prime,X^\prime-S^\prime) \\ (U,U -S) \ar[u]^{[ can\circ\sigma^N]}\ar@{-->}[ru]_{[r]} & (U,U -S)\ar@{-->}[ru]^{[l]} & (U^\prime,U^\prime -S^\prime)\ar[u]_{[can^\prime\circ \sigma^N ]}\ar[l]_{[\Pi]}
}
\end{gather*}
Since there is no one pair of composible dashed arrows in the following diagrams all of them are good with respect to a descent in sense of \ref{def:GammaGammaprimeGoodwrtDesctnt}.
\end{proof}

\begin{proof}[Proof of theorems \ref{th:InjLoc}-\ref{th:etex}]
The case of infinite base field for th.~\ref{th:InjLoc}-\ref{th:ZarexA1} is given by 
\cite[th. 2.11, th 2.9, th. 2.10]{GP-HIVth}.
The th.~\ref{th:etex} for the case of infinite field case is given by \cite[th. 2.13-2.14]{GP-HIVth}, and 
\cite[th 3.8]{DP-Sur2} (for the characteristic two case for the second claim)

Now by lemma \ref{lm:Diagrams} and corollary \ref{cor:FinDesInf} it follows that 
if for a field $k$ the theorems \ref{th:InjLoc}-\ref{th:etex} hold over any infinite field extension of $k$ then theorems \ref{th:InjLoc}-\ref{th:etex} hold over $k$.

In detail, by lemma \ref{lm:Diagrams} any theorem of the list \ref{th:InjLoc}-\ref{th:etex} 
states some weak lifting properties
in the categories $\overline{\mathbb ZF}_*$ (or $\overline{\mathbb ZF}^{pair}_*$).
By the results of \cite{GP-HIVth} the properties holds over any infinite field. 
Then by corollary \ref{cor:FinDesInf} any such a lifting property that holds over any infinite field holds over any finite field as well. So the claim follows.
\end{proof}
\newcommand{\scrF}{\mathcal F}
\begin{corollary}[see theorem 2.15 in \cite{GP-HIVth} for the infinite field case]\label{cor:HIFrProp}
Let $\scrF$ be a homotopy invariant stable linear framed presheaf.
Then the following properties holds:

1) Under the assumptions of theorem \ref{th:InjLoc} the homomorphisms $\eta^*\scrF(U)\to \scrF(k(X))$ and $(\eta^h)^*\scrF(U^h_x)\to \scrF(\Spec k(U^h_x)))$ are injective.

2) Under the assumptions of theorem \ref{th:InjA1}  the restriction homomorphism $\scrF(U)\to \scrF(V)$ is injective. 

3) Under the assumptions of theorem \ref{th:ZarexA1} the homomorphism 
$i^*\colon \scrF(U-S)/\scrF(U)\to \scrF(V-S)/\scrF(V)$
is an isomorphism;

4) Under the notation of theorem \ref{th:etex} the homomorphism
$pi^*\colon \scrF(U-S_x)/\scrF(U)\to \scrF(U^\prime-S^\prime_{x^\prime})/\scrF(U^\prime)$
is an isomorphism, where $S_x= \Spec \mathcal O_{S,x}$, $S^\prime_{x^\prime}= \Spec \mathcal O_{S^\prime,x^\prime}$.
\end{corollary}
\begin{proof}
The claim follows form theorems \ref{th:InjLoc}-\ref{th:etex} by the same argument as in \cite[theorem 2.15]{GP-HIVth}. 
\end{proof}

\begin{theorem}
\label{th:ZarrelexA1}
Let $X \in Sm_k$, $x \in X$ be a point, $W = \Spec(\mathcal O_{X,x})$. 
Let $i\colon V\subset \A^1_W$ be an open subset, $0_W\subset V$.
Then there are morphisms
$r \in \mathbb ZF_1^{pair}((\A_W^1 ,\A_W^1 -0\times W),(V,V -0\times W))$ and $l \in \mathbb ZF_1^{pair}((A_W^1 ,A_W^1 -0\times W),(V,V -0\times W))$ such that
$$[i]\circ [r]=[\sigma_{\A_W^1} ] \text{ and } [l]\circ[i]=[\sigma_V]$$
in $\overline{\mathbb ZF}^{pair}_1((\A_W^1 ,\A_W^1 -0\times W),(\A_W^1 ,\A_W^1 -0\times W))$ and $\overline{\mathbb ZF}^{pair}_1((V,V -0\times W),(V,V -0\times W))$ respectively.
\end{theorem}
\begin{proof}
We follow the scheme of the arguments of \cite{AD} or \cite{DrKold-ACorr} translated to the case of framed correspondences.

To prove the first claim 
we need to construct
\[r\in \mathbb ZF_1^{pair}( (\A^1_W,\A^1_W- 0_W) , (V,V- 0_W) ),
h_r\in \mathbb ZF_1^{pair}( (\A^1_W,\A_W^1- 0_W)\times\stackrel{\lambda}{\A^1}, (\A^1_W,\A^1_W- 0_W) )\]
such that $$h_r\circ i_0=i\circ r, \quad h_r\circ i_1=\id_{(\A^1_W,\A^1_W- 0_W)}.$$

Consider the following sections:
$$\begin{array}{lll}
s\in \Gamma(\PP^1_{W\times\stackrel{x}{\A^1}}, \mathcal O(n)) &
\tilde s\in \Gamma({\PP^1}_{W \times \stackrel{x}{\A^1}  \times \stackrel{\lambda}{\A^1} } , \mathcal O(n) ) &
s^\prime \in \Gamma( {\PP^1}_{W\times\stackrel{x}{\A^1} } , \mathcal O(n-1) ) \\
\tilde s\big|_{\PP^1\times W\times\A^1\times 0}=s & &
\tilde s\big|_{\PP^1\times W\times\A^1\times 1}=(t_0-x t_\infty)s^\prime
\\ 
s\big|_{((\PP^1\times W) \setminus V)\times\A^1 }= t_0^n &
\tilde s\big|_{\infty\times W\times\A^1\times\A^1 }= t_0^n &
s^\prime\big|_{\infty\times W\times\A^1 }= t_0^{n-1}
\\
s\big|_{0\times W\times\A^1} =t_\infty^{n-1} (t_0-x t_\infty) & 
\tilde s\big|_{0\times W\times\A^1\times\A^1} =t_\infty^{n-1}(t_0-x t_\infty) & 
s^\prime\big|_{0\times W\times\A^1}=t_\infty^{n-1}
\\
&& s^\prime\big|_{Z(t_0-x t_\infty)\times W}=t_\infty^{n-1}\\
\end{array}$$
where ${[t_0\colon t_\infty]}$ denotes coordinates on $\PP^1$. 
Such sections $s$ and $s^\prime$ exist for $n$ large enough by the Serre theorem \cite[Theorem 5.2]{Hartshorne-AlG} on sections of a powers of ample bundles, 
since $U$ is affine or local, and consequently $\mathcal O(1)$ is ample on $\PP^1\times U \times \A^1$ and $\PP^1\times U\times \A^1\times\A^1$. 
Having $s$ and $s'$, we then put $\tilde s=(1 -\lambda)s+\lambda (t_0-x t_\infty)s^\prime$.

Denote by $t=t_\infty/t_0$ the coordinate on $\A^1=\PP^1-\infty$.
Define the correspondence $r$ as a correspondence of pairs given by the pair of explicit framed correspondences
\[\begin{array}{llllll}
(\stackrel{t}{V}\times\stackrel{x}{\A^1}, & Z(s), & s/t_\infty^{n}, & pr_t)&\in& 
Fr_1(W\times\stackrel{x}{\A^1},\stackrel{t}{V}),\\
((\stackrel{t}{V}-0_W)\times(\stackrel{x}{\A^1}-0), & Z(s)\times_{\stackrel{x}{\A^1}}(\stackrel{x}{\A^1}-0), & 
s/t_\infty^{n}, & pr_t^\prime)&\in& 
Fr_1(W\times(\stackrel{x}{\A^1}-0),\stackrel{t}{V}-0_W),
\end{array}\]
where 
$$pr_t\colon \stackrel{t}{V}\times\stackrel{x}{\A^1}\to \stackrel{t}{V}, pr_t^\prime\colon (\stackrel{t}{V}-0)\times(\stackrel{x}{\A^1}-0)\to \stackrel{t}{V}-0_W$$ denote 
the projections.
Let us write in short that
$$r = [(\stackrel{t}{\A^1}\times W\times\stackrel{x}{\A^1}, Z(s), s/t_\infty^{n}, pr_t)]\in 
\mathbb ZF^{pair}( W\times (\stackrel{x}{\A^1},\stackrel{x}{\A^1}- 0) , (\stackrel{t}{V},\stackrel{t}{V}- 0_W) ).$$

In a similar way define the correspondence $h$
$$
h_{r,1} = [( \stackrel{t}{\A^1}\times W\times\stackrel{x}{\A^1}\times\stackrel{\lambda}{\A^1}, Z(\widetilde s), \widetilde s/t_\infty^{n}, pr_t^2 )] 
\in 
\mathbb ZF^{pair}( W\times(\stackrel{x}{\A^1}, \stackrel{x}{\A^1}-0)\times\stackrel{\lambda}{\A^1} ,W\times(\stackrel{t}{\A^1},\stackrel{t}{\A^1}-0) ),$$
where 
$pr_t^2$ is the projection $\stackrel{t}{\A^1}\times W\times\stackrel{x}{\A^1}\times\stackrel{\lambda}{\A^1}\to W\times \stackrel{t}{\A^1}$. 

Then the properties of $s$ and $\tilde s$ above implies that
\[\begin{aligned}
h_{r,1}\circ i_0=& i\circ r;\\
h_{r,1}\circ i_1=& [(\stackrel{t}{\A^1}\times W\times\stackrel{x}{\A^1}, Z(t-x), (t-x)g , pr_t^2)] + [(\stackrel{t}{\A^1}\times W\times\stackrel{x}{\A^1}, Z(g), (t-x)g , pr_t^2)]
\end{aligned}\]
where  $g= s^\prime/t_\infty^{n-1}\in k[\A^1\times\A^1\times U]$.
By the definition of the correspondences of pairs we see that the second summand is trivial in 
$\mathbb ZF^{pair}(W\times(\A^1,\A^1- 0)\times\stackrel{\lambda}{\A^1}, W\times(\A^1,\A^1- 0))$.
Now define the correspondence in
$\mathbb ZF^{pair}( W\times(\stackrel{x}{\A^1}, \stackrel{x}{\A^1}-0)\times\stackrel{\lambda}{\A^1}, W\times(\stackrel{t}{\A^1},\stackrel{t}{\A^1}-0) )$
\[h_{r,2} = 
[(\stackrel{t}{\A^1}\times W\times\stackrel{x}{\A^1}\times\stackrel{\lambda}{\A^1}, Z(t-x)\times\stackrel{\lambda}{\A^1}, (t-x)g(1-\lambda)+(t-x) , pr_t^2)]
\]
using the fact that $g\big|_{Z(t-x)}=1$.
Then 
\[\begin{aligned}
h_{r,2}\circ i_0=& [(\stackrel{t}{\A^1}\times W\times\stackrel{x}{\A^1}, Z(t-x), (t-x)g , pr_t^2)],\\
h_{r,2}\circ i_1=& [(\stackrel{t}{\A^1}\times W\times\stackrel{x}{\A^1}, Z(t-x), (t-x) , pr_t^2)] = id_{(\A^1_W,\A^1_W\setminus 0_W)} 
\end{aligned}\]

Thus we put 
$h = h_{r,1}+h_{r,2}-[(\stackrel{t}{\A^1}\times W\times\stackrel{x}{\A^1}, Z(t-x), (t-x)g , pr_t^2)]$,
and the claim follows.

2) 
To prove the second claim we need to construct
\[
l\in \mathbb ZF_1^{pair}( (\A^1_W,\A^1_W - 0_W) , (V,V - 0_W) ),
h_l\in \mathbb ZF_1^{pair}((V,V- 0_W)\times\A^1, (V,V- 0_W) )\]
such that $h_l\circ i_0=l\circ i$ and $h_l\circ i_1=\id_{(V,V\setminus 0_U)}$. 

For $n$ large enough similarly as above using the Serre theorem  \cite[Theorem 5.2]{Hartshorne-AlG}
we find the following sections: 
\[\begin{array}{lll}
s\in \Gamma(\stackrel{[t_0\colon t_\infty]}{ \PP^1}_{W\times \stackrel{x}{\A^1}}, \mathcal O(n) ) &
\tilde s\in \Gamma(\stackrel{[t_0\colon t_\infty]}{ \PP^1}_{ \stackrel{x}{V} \times \stackrel{\lambda}{\A^1} }, \mathcal O(n) ) &
s^\prime \in \Gamma( \stackrel{[t_0\colon t_\infty]}{ \PP^1}_{ \stackrel{x}{V} } , \mathcal O(n-1) ) \\
& \tilde s\big|_{\PP^1\times V\times 0}=f & \tilde s\big|_{\PP^1\times V\times 1}=(t_0-xt_\infty)s^\prime\\ 
s\big|_{((\PP^1\times W)\setminus V)\times\A^1 }= t_0^n & \tilde s\big|_{(\A^1\setminus V)\times V \times\A^1 }= t_0^n & g\big|_{(\A^1\setminus V)\times\A^1 }= t_0^n (t_0-x t_\infty)^{-1}\\
s\big|_{0\times W\times\A^1} =t_0-x t_\infty & \tilde s\big|_{0\times V\times \A^1} =t_0-x t_\infty & s^\prime\big|_{0\times V}=t_\infty^n\\
&& s^\prime\big|_{Z(t_0-x t_\infty)\times W}=t_\infty^{n-1}
, \end{array}\]
where $g= s^\prime/t_\infty^{n-1}\in k[\A^1\times V]$. 

Next, under the same notation as above in the point (1) of the proof define
\begin{align*}
l &= (\stackrel{t}{\A^1}\times W\times\stackrel{x}{\A^1}, Z(s), s/t_\infty^{n}, pr_t) 
\in \mathbb ZF^{pair}( W\times (\stackrel{x}{\A^1},\stackrel{x}{\A^1}- 0), (\stackrel{t}{V},\stackrel{t}{V}- 0_W) ),\\ 
h_{l,1} &= (\stackrel{t}{V} \times_W \stackrel{x}{V}\times\stackrel{\lambda}{\A^1}, Z(\tilde s)\times_{\stackrel{x}{\A^1}} V, \tilde s/t_\infty^{n} , pr^3_t)
\in \mathbb ZF^{pair}((\stackrel{x}{V},\stackrel{x}{V}-0_W)\times\stackrel{\lambda}{\A^1}, (\stackrel{x}{V},\stackrel{x}{V}- 0_W) ),
\end{align*}
where $t=t_0/t_\infty$ denotes the coordinate on $\A^1=\PP^1-0$,
and 
$pr^3_t\colon \stackrel{t}{V} \times_W \stackrel{x}{V}\times\stackrel{\lambda}{\A^1} \to \stackrel{t}{V}$ is the projection.

Then the properties of $s$ and $s'$ above imply that 
\begin{multline*}
h_l\circ i_0=i\circ l;
h_l\circ i_1=[( \stackrel{t}{\A^1}\times W\times\stackrel{x}{\A^1}\times\stackrel{\lambda}{\A^1}, Z(t-x), (t-x)g ,pr_t )] + \\ [( \stackrel{t}{\A^1}\times W\times\stackrel{x}{\A^1}\times\stackrel{\lambda}{\A^1}, Z(g), (y-x)g , pr^2_t )].
\end{multline*}
The second summand is trivial by the definition of the group $\mathbb ZF(W\times({\A^1},{\A^1}-0)\times{\A^1}, W\times({\A^1},{\A^1}-0))$.
Then since $g\big|_{Z(t-x)}=1$ we can define 
\begin{multline*}
h_{r,2} = 
[(\stackrel{t}{V}\times_W\stackrel{x}{V}\times\stackrel{\lambda}{\A^1}, Z(t-x)\times\stackrel{\lambda}{\A^1}, (t-x)g(1-\lambda)+(t-x) , pr_t^2)] \in \\
\mathbb ZF^{pair}( (\stackrel{x}{V}, \stackrel{x}{V}-0_W)\times\stackrel{\lambda}{\A^1}, (\stackrel{t}{V},\stackrel{t}{V}-0) ).
\end{multline*}
Finally we put 
$h = h_{r,1}+h_{r,2}-[(\stackrel{t}{V}\times_W\stackrel{x}{V}, Z(t-x), (t-x)g , pr_t^2)]$,
and then
\begin{multline*}
h_{r}\circ i_0= l;\quad 
h_{r}\circ i_1 = [(\stackrel{t}{V}\times_W\stackrel{x}{V}, Z(t-x), (t-x)g , pr_t^2)]+\\
[(\stackrel{t}{V}\times_W\stackrel{x}{V}, Z(g), (t-x)g , pr_t^2)] -
[(\stackrel{t}{V}\times_W\stackrel{x}{V}, Z(t-x), (t-x)g , pr_t^2)]+\\
[(\stackrel{t}{V}\times_W\stackrel{x}{V}, Z(t-x), (t-x), pr_t^2)] = id_{(V,V- 0_U)}
\end{multline*}
So the claim is done.
\end{proof}
%
\begin{corollary}[see corollary 2.16 in \cite{GP-HIVth} for the infinite base field case for the second claim]
\label{cor:Cor216GP14}
Suppose that $W\in\Sm_k$ is an affine scheme or a local scheme, and let $V_1\subseteq V_2\subseteq\A^1_W$ be a pair of open subschemes such that $0_W\in V_1$. Let $i\colon V_1\subseteq V_2$ denote the inclusion.
Then, for any homotopy invariant stable linear framed presheaf $\scrF$,
the restriction homomorphism $i^*$ induces an isomorphism
\[i^*\colon \scrF (V_2\setminus 0_W)/\scrF (V_2)\xrightarrow{\cong}\scrF (V_1\setminus 0_W)/\scrF (V_1).\]
And consequently 
\[\scrF (\A^1_W\setminus 0_W)/\scrF (\A^1_W)\xrightarrow{\cong}\scrF (V\setminus 0_W)/\scrF (V).\]
where $V = (\A^1_W)_{0_W}$ is the local scheme corresponding to the closed point of the subscheme $0_W\subset \A^1_W$.
\end{corollary}
\begin{proof}
It follows form the theorem \ref{th:ZarrelexA1} that 
$$\scrF (\A^1_W- 0_W)/\scrF (\A^1_W)\xrightarrow{\cong}\scrF (V^\prime\setminus 0_W)/\scrF (V^\prime),$$
for any open $V^\prime\subset \A^1_W$, $0_W\subset V$.
Now the first claim follows
since we have
$$\scrF (V_1\setminus 0_W)/\scrF (V_2)\simeq \scrF (\A^1_W- 0_W)/\scrF (\A^1_W)\xrightarrow{\cong}\scrF (V_1^\prime\setminus 0_W)/\scrF (V_1^\prime).$$
The second claim follows, since $$\scrF (V\setminus 0_W)/\scrF (V) = \varinjlim_{V^\prime\subset\A^1_W,0_W\subset V^\prime} \scrF (V^\prime\setminus 0_W)/\scrF (V^\prime).$$
\end{proof}

\section{Framed motives over finite fields}\label{sect:StrHIFrFinF}

In this section we apply the finite descent from the previous section to 
extend the theory of framed motives \cite{GP_MFrAlgVar} to the finite base field case.
The results of the theory of framed motives are
based on sequence on theorems about framed correspondences (and corresponding $S^1$-spectra of representable sheaves), 
and on (pre)sheaves with framed transfers
including 
the strictly homotopy invariance theorem proven in \cite{GP-HIVth}, 
cancellation theorem \cite{AGP-FrCanc}, 
and the so called 'cone' theorem \cite{GNP_FrMotiveRelSphere}.
The assumption on the base field to be infinite is needed 
in the strictly homotopy invariance theorem and cancellation theorem.
So to extend the results of the theory to the case if finite fields it is enough to prove the mentioned theorems for the case such fields
Let us recall this theorems.

\subsection{}

In \cite{GP_MFrAlgVar} the theory of framed motives is constructed over an infinite perfect field of characteristic different form 2.
In this section we explain how the result of the previous section extend the theory to the finite base field case.

The reason of the restrictive assumption on the base field in \cite{GP_MFrAlgVar} are the strictly homotopy invariance \cite{GP-HIVth} and cancellation theorems \cite{AGP-FrCanc}, and up to the references to these results the assumptions on the base field are not needed. 
Moreover the proof of the cancellation theorem \cite{AGP-FrCanc} 
does not use the infiniteness assumption except the references to \cite{GP-HIVth}.


In its turn for an arbitrary perfect field the strictly homotopy invariance by the arguments of \cite{GP-HIVth} 
follows form the injectivity and excision theorems \cite[theorems 2.9, 2.10, 2.11, 2.13, 2.14, 2.15]{GP-HIVth}.
This means that the text of the proofs in \cite{GP-HIVth} uses the infiniteness assumption nowhere except the references to the injectivity and excision theorems \cite[theorems 2.9, 2.10, 2.11, 2.13, 2.14, 2.15]{GP-HIVth},.

So whenever we have proven the injectivity and excision properties over finite fields, see ths \ref{th:InjLoc}-\ref{th:etex} and corollary \ref{cor:HIFrProp},
we have got the strictly homotopy invariance and cancellation theorems consequently the 
results on framed motives \cite{GP_MFrAlgVar}. 

In what follows we recall the list of steps in the mentioned above reductions and give precise references to the arguments in the mentioned sources. 
All the proofs form \cite{GP-HIVth} and \cite{AGP-FrCanc} we refer to hold word by word in the general prefect base field case.

\subsection{Strictly homotopy invariance}
\begin{theorem}[strictly homotopy invariance, see theorem 1.1 in \cite{GP-HIVth} for infinite base fields]\label{th:StrHimInv}
Any homotopy invariant linear framed $\sigma$-stable presheaf $F$ over a perfect field $k$ is strictly homotopy invariant and $\sigma$-stable.
\end{theorem}
\begin{proof}
As shown in \cite[section 16]{GP-HIVth}
the claim follows 
from corollary \ref{cor:HIFrProp} and corollary \ref{cor:Cor216GP14}.

\end{proof}
\begin{remark}
Let us list the sequence of lemmas and propositions
by which \cite[]{} is proven:
Namely, it is 
reduction is given by the sequences of 
\cite[16.9 16.8 16.7 16.6 16.4 16.3 16.2]{GP-HIVth}
and \cite[16.12 16.4 16.3]{GP-HIVth}.
\end{remark}

\subsection{Cancellation theorem}

\begin{theorem}[linear cancellation theorem, see theorem C in \cite{AGP-FrCanc} for infinite base fields]\label{th:canc}
For a prefect base field $k$ the natural homomorphism of complexes of abelian groups
\begin{equation}\label{eq:Canc}
\mathbb ZF(X\times\Delta^\bullet,Y)\to \mathbb ZF(X\times\Delta^\bullet\wedge (\Gm,1),Y\wedge (\Gm,1))
\end{equation}
is an quasi-isomorphism.
\end{theorem}
\begin{proof}
The claim follows 
from theorems \ref{th:StrHimInv} and \ref{th:InjLoc}-\ref{th:etex}
by the arguments as in \cite[theorem C]{AGP-FrCanc}.
\end{proof}

\subsection{Remarks on more precise and  shorter arguments.}

\begin{remark}
The infiniteness assumption on the base field in \cite{GP-HIVth}
actually is needed only in theorems 2.11, 2.13, 2.14, and corollary 2.16.
So we could not to reprove theorems \ref{th:InjA1}, \ref{th:ZarexA1} that are \cite[th 2.9, th 2.10]{GP-HIVth}.
\end{remark}

\begin{remark}
In the above argument we have extended \cite[th 2.9, th. 2.10, th 2.11]{GP-HIVth} to the finite field case applying corollary \ref{cor:FinDesInf}, and 
we have gave the independent argument for \cite[corollary 5]{GP-HIVth} for the case of such fields.
In the same time let us note that if we reformulate \cite[corollary 5]{GP-HIVth} in therms of the category of correspondences of pairs, in a similar form to the theorem \ref{th:ZarrelexA1} or \cite[theorem 2.12]{GP-HIVth}, then \cite[corollary 5]{GP-HIVth} could deduced in the case of finite fields form the infinite field case applying corollary \ref{cor:FinDesInf} as well.
\end{remark}

\begin{remark}
Using corollary \ref{cor:FinDesN} instead of
corollary \ref{cor:FinDesInf} theorems \ref{th:InjLoc} and \ref{th:etex} could be deduced from the inner results of \cite{GP-HIVth} . 
Namely we mean some inner statements in the proof of theorems \ref{th:InjLoc} and \ref{th:etex} in \cite{GP-HIVth}.
The corollary \ref{cor:FinDesInf} and a short proof above was explained to the authors by I.~Panin and A.~Tsybyshev.

Let us explain this argument.
The assumption in the proofs is needed to satisfy some conditions of generic position.
Namely,
\begin{itemize}
\item[(1)] constructing a relative curve with a ``good'' compactification in both theorems it is needed to choose some projection in affine space such that the restriction to some smooth subscheme of codimension one is \'{e}tale;
\item[(2)] constructing of the morphism from the \'{e}tale neighbourhood $\mathcal V$ of the support of the framed correspondence to the target $Y$ of the framed correspondence, it is assumed to choose a generic projection again;
\item[(3)] in the injective \'{e}tale excision, when choosing a section of a line bundle on a projective curve that does not satisfy some closed property.
\end{itemize}
All these constructions require to find a $k$-rational point in a non-empty open subscheme $U\subset \PP^N_k$ for some $N$.
By lemma \ref{lm:bigK-Aff} such point exists for all enough big field extensions $K/k$, and so theorem \ref{th:etex} holds
for all fields $K$, $k\subset K$, such that such that $\deg K/k>L$ for some $L \gg 0$.
Now the case of a finite base field $k$ follows by lemma \ref{cor:FinDesN}. 
\end{remark}

\begin{remark}
The finite descent argument allow us to deduce theorem \ref{th:canc} over finite base fields from the infinite base field case directly without references to the inner scheme of the proof in \cite{AGP-FrCanc}.
We show the surjectivity of the homomorphism \ref{eq:Canc}, 
the injectivity is similar and simpler.

Let $c\in \mathbb ZF(X\times\Delta^i\wedge (\Gm,1),Y\wedge (\Gm,1))$ be a cycle of the complex $\mathbb ZF(X\times\Delta^\bullet\wedge (\Gm,1),Y\wedge (\Gm,1))$. We need to show that there is a cycle $c^\prime\in \mathbb ZF(X\times\Delta^i,Y)$ of the complex $\mathbb ZF(X\times\Delta^\bullet,Y)$,
and an element $b\in \mathbb ZF(X\times\Delta^{i-1}\wedge (\Gm,1),Y\wedge (\Gm,1))$ 
such that $c = \theta(c^\prime)+d_i(b)$.
By assumption such elements $c^\prime_K$ and $b_K$ exist over an infinite field extension $K/k$.
As in Corollary \ref{cor:FinDesInf} this implies that there exist $c^\prime_{K_1}$, $b_{K_1}$, $c^\prime_{K_2}$ and $b_{K_2}$ for a pair of co-prime finite field extensions $K_1/k$, and $K_2/k$.
Now we see that elements $c^\prime  = pr \circ (c_{K_1}\amalg c_{K_2}) \circ (L\boxtimes id_{X\times\Delta^{i}})$ and $b= pr^\prime \circ (b_{K_1}\amalg b_{K_2}) \circ (L\boxtimes id_{X\times\Delta^{i-1}})$ satisfy the required conditions, 
where $pr\colon Y\times S\to Y$, $pr^\prime \colon Y\wedge (\Gm,1)\times S\to Y\wedge (\Gm,1)$, $S=\Spec{K_1}\amalg \Spec K_2$, 
and $L\in \mathbb ZF_1(\Spec k,\Spec{K_1}\amalg \Spec K_2)$ is defined as in \eqref{eq:invlift}.
\end{remark}
\vspace{5pt}

\subsection{Corollary on zeroth motivic homotopy groups} 

By the above all results of \cite{GP_MFrAlgVar} hold over an arbitrary perfect field. 
A particular consequence of 
\cite[theorem 11.7]{GP_MFrAlgVar} 
is the following result on the zeroth motivic homotopy groups
%
\begin{theorem}\label{th:piSH=H0Fr}
Let $k$ be a perfect field, 
then
$$[pt_+,\G^{\wedge n}_m]_{\SHd(k)}\simeq H^0(\mathbb ZF(\Delta_k^\bullet, \G^{\wedge n}_m)).$$
\end{theorem}

\section{Proof of the isomorphism $\mathrm{K}^\mathrm{MW}_n\simeq H^0(\mathbb ZF(\Delta^\bullet, \G^{\wedge n}_m))$.}
\label{sect:KMWHoFr}
In this section we recall Neshitov's proof of the isomorphism \eqref{eq:KMWHoZF} in the case characteristic zero \cite{Nesh-FrKMW},
and extend the arguments to obtain a proof for perfect fields of characteristic different from two.
Actually, here we repeat the same arguments as in \cite{Nesh-FrKMW} replacing some references to the results that require alternative proof.
Namely the reference to the proof of Steinberg relation given in \cite[subsection 8.3]{Nesh-FrKMW} is replaced by the reference to \cite{PHandIK-Steinberg}, \cite{Powell-Steinberg} or \cite{DKMWhom}; and 
the reference to the moving lemma \cite[lemma 4.10]{Nesh-FrKMW} is replaced by the lemma \ref{lm:iOScormove} proven in the next section of the present article.

It is written at the beginning of \cite{Nesh-FrKMW} that throughout the text the base field is of characteristic $0$. In the same time the assumption is used only in few places and many of the arguments works in any characteristic or in any odd characteristic.
Let us cite some of original lemmas from \cite{Nesh-FrKMW} indicating what are essential assumptions on the base field in the proofs for each statement.

\begin{theorem}
For a perfect field $k$ of characteristic different form $2$ there is a graded ring isomorphism 
\begin{equation}\label{eq:KMWHoZF}
\mathrm{K}^\mathrm{MW}_{\geq 0}(k)\simeq H^0(\mathbb ZF(\Delta^\bullet_k, \G^{\wedge *}_m))
\end{equation}
that takes the symbol $[a_1,\dots a_n]\in \mathrm{K}^\mathrm{MW}_n(k)$, $a_i\in k^\times$, to the class of the correspondences defined by the regular map $(a_1,\dots a_n)\colon \pt_k\to \G_m^{\wedge n}$.
\end{theorem}
\begin{proof}
The lemmas \ref{lm:Phi}, \ref{lm:Psi} give us the injective ring homomorphism $\mathrm{K}^\mathrm{MW}_n(k)\to H^0(\mathbb ZF(\Delta^\bullet_k, \G^{\wedge n}_m))$. The homomorphism is surjective by lemma \ref{lm:sur}.
\end{proof}

\begin{lemma}[section 6.2, lemma 9.1, section 7 in \cite{Nesh-FrKMW}]\label{lm:Phi}
For a perfect field $k$, $\chark k\neq 2$, 
there are homomorphisms 
$$\Phi_{n,k}\colon H^0(\mathbb ZF(\Delta^\bullet,\G_m^{\wedge n}))\to \mathrm{K}^\mathrm{MW}_n(k), n\geq 0,$$ 
that takes the class of the correspondences defined by the map $(a_1,\dots a_n)\colon \pt_k\to \G_m^{\wedge n}$ to the symbol $[a_1,\dots a_n]$.
\end{lemma}
%

\begin{lemma}\label{lm:Psi}
For a field $k$, $\chark k\neq 2$, there is a graded ring homomorphism 
$$\Psi_{*,k}\colon \mathrm{K}^\mathrm{MW}_{\geq 0}(k)\to H^0(\mathbb ZF(\Delta^\bullet_k,\G_m^{\wedge *}))$$
that 
takes the symbol $[a_1,\dots a_n]\in \mathrm{K}^\mathrm{MW}_n(k)$, $a_i\in k^\times$, to the class of the correspondences defined by the regular map $(a_1,\dots a_n)\colon \pt_k\to \G_m^{\wedge n}$.
\end{lemma}
The arguments of \cite[section 7, section 8]{Nesh-FrKMW} proves the claim for a field $k$, $\chark k \neq 2$, $\chark k\neq 3$.
The assumption $\chark k\neq 2$ is needed for the proof of the relations in $GW(k)$ \cite[section 7]{Nesh-FrKMW}, and the assumption $\chark k\neq 3$ for the Steinberg relation \cite[subsection 8.2]{Nesh-FrKMW}. 
Nevertheless theorem \ref{th:piSH=H0Fr} provides that the original work \cite{PHandIK-Steinberg}, where the Steinberg relation is proven in $\pi^{2,2}(S)$ over an arbitrary base scheme $S$, implies the relation in $H^0(\mathbb ZF(\Delta^\bullet_k,\G_m^{\wedge 2}))$ for a prefect $k$. Thus repeating the arguments replacing the proof of Steinberg relation by the reference to \cite{PHandIK-Steinberg} we get the construction of $\Psi_*$ over a perfect fields of odd characteristic.
Let us note that as shown 
in the preprint \cite{DKMWhom}
the required homomorphism of rings $\Psi_*$ actually exists over an arbitrary base scheme. 
\begin{proof}[Proof of lemma \ref{lm:Psi}]

To construct the homomorphism it is needed to prove the relations of the Milnor-Witt K-theory ring $\mathrm{K}^\mathrm{MW}_*(k)$
in $H^0(\mathbb ZF(\Delta^\bullet,\G_m^{\wedge *}))$.
Due to theorem \ref{th:piSH=H0Fr} this is equivalent to prove the relations in the ring $\pi^{*,*}_s(\pt_k)$.

Firstly, the arguments of \cite[lemma 7.6]{Nesh-FrKMW} provides the homomorphism $\Psi_0\colon GW(k)\to H^0(\mathbb ZF(\Delta^\bullet,\pt_k)$. The moving lemma \ref{lm:standFrptpt} proven in the next section yields immediate that $\Psi_0$ is surjective. Then by lemma \ref{lm:Phi} it follows that $\Psi_0$ is an isomorphism.

As shown in \cite[subsection 8.3]{Nesh-FrKMW}
the isomorphism $\Psi_0$ and the Steinberg relation 
implies the rest relations of $\mathrm{K}^\mathrm{MW}$.
In detail, \cite[lemmas 8.5, corollary 8.14, lemma 8.15]{Nesh-FrKMW} provides  the homomorphism $\Psi_1\colon\mathrm{K}^\mathrm{MW}_1(k)\to H^0(\mathbb ZF(\Delta^\bullet,\G_m^{\wedge 1}))$.
Hence since by \cite[remark 3.2]{M-A1Top} the Steinberg relation defines a factor algebra of the tensor algebra $T_{GW(k)}(\mathrm{K}^\mathrm{MW}_1(k))$ 
we get the homomorphism $\Psi_*$.
\end{proof}
%
\begin{lemma}[proposition 9.6 \cite{Nesh-FrKMW} for $\chark k=0$]\label{lm:sur}
The homomorphism $\Psi_{*,k}$ is surjective for any perfect field $k$, $\chark k\neq 2$. 
\end{lemma}
\begin{proof}
As noted above the case of $\Psi_0$ follows from the moving lemma proven in the next section \ref{lm:standFrptpt}.
The general case follows 
form lemmas \ref{lm:Tr} and lemma \ref{lm:planRC:simplFrcor}
in a similar way to \cite[proposition 9.6]{Nesh-FrKMW}. 
\end{proof}

%

\begin{lemma}[see lemma 9.5 and subsection 3.1 in \cite{Nesh-FrKMW}]\label{lm:Tr}
Let $k$ be a field, $\chark k\neq 2$.
Let $L/k$ be a finite field extension of a field $k$, $\chark k\neq 2$. 
There are a transfer map $Tr^L_k \colon \mathrm{K}^\mathrm{MW}_n(L) \to 
\mathrm{K}^\mathrm{MW}_n(k)$ given by \cite[definition 4.26]{M-A1Top} and \cite[theorem 4.27]{M-A1Top} and a transfer map 
$tr^L_k\colon H^0(\mathbb ZF(\Delta^\bullet_L ,\G_m^{\wedge n}))\to H^0(\mathbb ZF(\Delta^\bullet_k ,\G_m^{\wedge n}))$
such that
$tr^L_k\circ \Psi_{n,L} = \Psi_{n,k}\circ Tr^L_k$ for all integer $n\geq 0$.
\end{lemma}
\begin{proof}
The claim follows similar as in \cite[lemma 9.5]{Nesh-FrKMW}
using lemmas \ref{lm:TrDiagGW} and \ref{lm:GWKS} and the Steinberg relation.
\end{proof}

\begin{lemma}[Lemma 7.9 \cite{Nesh-FrKMW}.]\label{lm:TrDiagGW}
Let $k$ be a field, $\chark k\neq 2$. Suppose $L/k$ is separable field extension, $[L\colon k]=l$ is a prime number, and $k$ has no prime-to-$l$ extensions.
Then the transfer diagram is commutative
$$\xymatrix{
\mathrm{K}^\mathrm{MW}_0(L)\ar[d]^{Tr^L_k} \ar[r]^{\Psi_0} & H^0(\mathbb ZF(\Delta_L^\bullet,\pt_L)\ar[d]^{tr_{L/k}} \\
\mathrm{K}^\mathrm{MW}_0(k)\ar[r]^{\Psi_0} & H^0(\mathbb ZF(\Delta_k^\bullet,\pt_k)
}$$
where $Tr^L_k$ is the transfer map homomorphism fro Milnor-Witt K-theory defined in \cite[Definition 4.28]{M-A1Top}.
\end{lemma}
\begin{proof}
We repeat the proof from \cite{Nesh-FrKMW}.

Take $L=k(\alpha)$. Since $l$ is prime, we may assume that $L=k(\alpha)$.
Then $tr_{L/k}(\Psi(\langle\alpha\rangle))=tr_{L/k}(\langle\alpha\rangle)=(\A^1_L,\alpha(x-\alpha),pr_k)$. 
Let $p(x)$ be the minimal monic polynomial of $\alpha$.
Then $(\A^1_L,\alpha(x-\alpha),pr_k)\sim (U^\prime, xp^\prime(x)p(x), pr_k)$, where $U^\prime\subset \A^1_L$ is an open subset does not contain any root of $xp^\prime(x)p(x)$ except $\alpha$.
Take $U\subset \A^1_k$ to be the image of $U^\prime$.
Then $tr_{L/k}(\Psi(\langle\alpha\rangle))=(U^\prime, xp^\prime(x)p(x), pr_k)=(U, xp^\prime(x)p(x), pr_k)=\langle xp^\prime(x)p(x)\rangle - (\A^1_k-Z(p), xp^\prime(x)p(x), pr_k) = \langle c_d x^{l+d+1}\rangle - (\A^1_k-Z(p), xp^\prime(x)p(x), pr_k)$, where $d=\deg p^\prime(x)<l$, and $c_d$ is the leading coefficient of $p^\prime(x)$. Since $k$ has no prime-to-$l$ extensions, all roots of $p^\prime(x)$ are rational, so $xp^\prime(x)=c_d(x-\lambda_1)^{r_1}(x-\lambda_g)^{r_g}$.
Then by \cite[Lemma 7.9]{Nesh-FrKMW} 
$$
\Phi_0(tr_{L/k}\Psi_0(\langle \alpha \rangle)) = 
\langle c_d\rangle s_\varepsilon - \big(\sum\limits_{i=1}^g (r_i)_\epsilon \langle c_d p(\lambda_i) \lambda_i \prod\limits_{j\neq i} (\lambda_i-\lambda_j)^{r_j}\rangle\big), 
s = 1+d+l.
$$
Let $\tau^L_k$ be the geometrical transfer map defined in \cite[4.2]{M-A1Top}.
Then
$
Tr^L_k(\langle\alpha\rangle) = \tau^L_k(\alpha)(\langle p^\prime(\alpha)\alpha \rangle) = -\partial_\infty^{-1/x}([xp^\prime(x)p(x)]- \sum_{i=1}^g [(x-\lambda_i)^{r_i} c_d p(\lambda_i)\lambda_i\prod_{j\neq i} (\lambda_i- \lambda_j)^{r_j} ]) =  \langle c_d\rangle s_\epsilon  - \big(\sum_{i=1}^g (r_i)_\epsilon \langle c_d p(\lambda_i)\lambda_i\prod_{j\neq i} (\lambda_i- \lambda_j)^{r_j} \rangle\big)
$ by \cite[Lemma 7.9]{Nesh-FrKMW}.
Thus $\Phi_0(tr_{L/k}\Psi_0(\langle\alpha\rangle))=Tr^L_k(\alpha)$. Hence the diagram commutes since $\Psi_0\circ \Phi_0$ is identity.
\end{proof}

\begin{lemma}[lemma 9.3 \cite{Nesh-FrKMW}]\label{lm:GWKS}
Let $k$ be a field, $\chark k\neq 2$.
Suppose $L/k$ is separable field extension, $[L\colon k]=l$ is a prime number, and $k$ has no prime-to-$l$ extensions.
Then $ $ is generated as abelian group by $GW(L) \mathrm{K}^\mathrm{MW}_1(k)+S$, where $S=\{\langle\pm\omega_0(a)\rangle[a] | a\in L^\times\}$, where for any $a\in \L^\times$ $\omega_0(a)=p^\prime(a)$ and $p^\prime$ is a derivative of the minimal polynomial of $a$.
\end{lemma}
\begin{proof}
The proof is the same as for \cite[lemma 9.3]{Nesh-FrKMW} just without the equality $d+m=l-1$ in the fourth row of the proof.
\end{proof}

\begin{lemma}\label{lm:planRC:simplFrcor}
For any element in $\mathbb ZF(\pt_k, \G^{\wedge n}_m)$ the class $[c]\in H^0(\mathbb Z(\Delta^\bullet_k,\G_m^{\wedge n}))$ is equivalent to the class of 
some $\tilde c = \in \mathbb ZF_1(\pt_k, \G^{\wedge n}_m)$ 
such that the support of $c$ 
is smooth. %
\end{lemma}
\begin{proof}
Any framed correspondences is equivalent to a correspondences in 
$\mathbb ZF(\pt_k, \G^{\wedge n}_m)$ 
that support is smooth,
by lemma \ref{lm:iOScormove} (with $i=0$), 
which is proven in the next section.
In other words this means that the support is a set of points (with the separable residue fields).
Now since any simple correspondences in $\mathbb ZF(\pt_k, \G^{\wedge n}_m)$ 
is equivalent to a simple one in $\mathbb ZF_1(\pt_k, \G^{\wedge n}_m)$ by \cite[Lemma 4.10]{Nesh-FrKMW}
the claim follows.
\end{proof}

\section{Moving lemma}\label{sect:MoovLm}

In this section we prove Lemma \ref{lm:planRC:simplFrcor}. 
Throughout we assume that $k$ is perfect. 
\begin{definition}\label{def:iSmooth}
Let
$c=(Z\subset \A^n,v\colon \mathcal V\to \A^n, \varphi=(\varphi_i)\in k[\mathcal V]^n,g\colon \mathcal V\to Y)\in Fr_n(\pt_k,Y)$
be a framed correspondence such that 
$v$ is an open immersion, and $Y\subset \A^e$ is an open subscheme.
Then $c$ is said to be an $(i)$-\emph{simple} correspondence for $i=1,\dots n$ iff
there is a vector of sections $(s_j)_j, s_j\in \Gamma(\P^n,\cO(d_j))$,
such that 
\begin{itemize}
\item[1)]
$v^*(s_j/t_\infty^{d_j})\big|_{Z(I(Z)^2)}=\varphi_j\big|_{Z(I(Z)^2)}$,
$s_j\big|_{\PP^{n-1}}=t_j^{d_j}$, $j=1,\dots,n$,
\item[2)]
$Z\subset \P^n - B_i$, where $B_i=\bigcup_{1\leq j< i} \Sing Z_{red}(s_1,\dots ,s_j)$,
and 
\item[3)]
$Z_{red}(s_1,\dots,s_{i-1})\cap Z(s_{i},\dots,s_n)$ is smooth.
\end{itemize}
\end{definition}
\begin{remark}\label{rm:siml0}
Because of the condition (3), 
any $(1)$-simple correspondences $c$ is simple, i.e., the (non-reduced) support of $c$ is smooth. 
\end{remark}

\begin{lemma}\label{prelm:GSO}
Suppose $k$ is perfect. Let $s=(s_i)_i$ 
be a vector of sections 
$s_i\in \Gamma(\PP^n, \cO(d_i))$ such that
$s_i\big|_{\PP^{n-1}}=t_i^{d_i}$, and denote $Z=Z(s)$.
Then there is a vector of sections $\os=(\os_i)$, $\os_i\in \Gamma(\P^n,\cO(d_i^\prime))$ 
such that $\os_i\big|_{\PP^{n-1}}=t_i^{d_i^\prime}$, $\os\big|_{\dub{Z(s)} }= (s t_\infty^{d_i^\prime-d_i})\big|_{\dub{Z(s)} }$, %
and such that 
$Z_{red}(\os_1,\dots \os_{n-1})\not\subset B_n$, where $$B_n= \bigcup\limits_{1\leq i< n}\Sing Z_{red}(\os_1,\dots \os_i).$$ %
Here $\PP^{n-1} \subset \PP^{n}$ is the subspace at infinity and $t_\infty\in \cO(1)$, $Z(t_\infty)=\P^{n-1}$.
\end{lemma}
\begin{proof}
Firstly we will prove that $\forall l\in \mathbb Z$, $1\leq l\leq n$,
there is a 
vector of sections $\os=(\os_i)_{i=1\dots n}$, $\os_i\in \Gamma(\P^n,\cO(d_i^\prime))$ 
such that $\os_i\big|_{\dub{Z(s)} }= (s_i t_\infty^{d_i^\prime-d_i})\big|_{\dub{Z(s)} }$, $\os_i\big|_{\PP^{n-1}}=t_i^{d_i^\prime}$ and such that 
\begin{equation}\label{eq:sectGensmCond} 
Z_{red}(\os_1,\dots \os_{l-1})\not\subset B_{l-1},  \;\text{where } B_{l-1}=\bigcup_{1\leq i< l-1} \Sing Z_{red}(\os_1,\dots \os_i)
\end{equation}

The proof is by induction on $l\in \mathbb Z$. The base case $l=1$ is clear.
Assume inductively we have proven \eqref{eq:sectGensmCond} for some $l$.
We will prove the claim for $l+1$.

Let $(s_i)$ be a vector of sections such that condition \eqref{eq:sectGensmCond} holds. %
Now we need to construct a section $\os_l\in \Gamma(\P^1, \cO(d_l^\prime))$ 
such that 
$\os_l\big|_{\PP^{n-1}}=t_l^{d_l}$,
$\os_l/t_\infty^{d_l^\prime}\big|_{\dub{Z(s)}}=s_l/t_\infty^{d_l}\big|_{\dub{Z(s)}}$, 
and such that
$$Z_{red}(\os_1,\dots \os_{l-1},\os_l)\not\subset B_{l} = \bigcup_{1\leq i< l} \Sing Z_{red}(\os_1,\dots \os_i). $$

Consider the scheme $Z_{l-1}=Z_{red}(\os_1,\dots \os_{l-1})$ and the reduced closed subschemes 
$B_{l-1}^\prime= Z_{l-1}\cap B_{l-1}$, $B_{l}^\prime= Z_{l-1}\cap B_{l}$.
Note that since $s_i\big|_{\P^{n-1}}=t_i^{d_i}$, it follows that $Z_{l-1}$ is of pure dimension $n-l+1$.
Since $k$ is perfect and $Z_l$ is reduced $Z_{l-1}-\Sing(Z_{l-1})$ is open dense in $Z_{l-1}$.
Recall that by the induction assumption $ Z_{l-1}\not\subset B_{l-1}$.
Then $Z_{l-1}\not\subset B^\prime_{l-1}$,
and by the above this implies that $Z_{l-1}-\Sing(Z_{l-1})\not \subset B^\prime_{l-1}$.
Hence $Z_{l-1}\not\subset B^\prime_{l-1}\cup \Sing Z_{l-1} = B^\prime_{l}$.
Thus $B^\prime_{l}$ is a proper subset of $Z_{l-1}$,
and so $\dim B_{l} \leq n-l$.

Denote $B_{l}^{\prime\prime}= (B_{l}^\prime)_{red}$. Let $P\subset Z_{l-1}$ be a finite set of points in $Z_{l-1}-(Z\cup \P^{n-1})$ such that $P$ contains at least one point in each irreducible component of $B_{l}^{\prime\prime}-(Z(s)\cup \P^{n-1})$. %
Using Serre's theorem \cite[Theorem 5.2]{Hartshorne-AlG} we can choose a section $\os_l\in \Gamma(\mathbb P^n,\mathcal O(d^\prime_l))$ for some $d_i^\prime\in \mathbb Z$ such that 
$\os_l\big|_{\dub{ Z(s) } }=s_l t_{\infty}^{d^\prime_l-d_l}\big|_{\dub{ Z(s) } }$,
$\os_l\big|_{\PP^{n-1}}=t_l^{d_i^\prime}$,
and $\os_i$ is invertible on $P$.

Denote $Z_{l}=Z_{l-1}\cap Z(\os_l)$.
To get the claim of the induction step we have to show that $Z_{l}\not\subset B_{l}$.
By the above 
$\dim Z_{l} \geq \dim Z_{l-1}-1=n-l\geq\dim B_{l}^{\prime\prime}$. %
Hence if $Z_{l}\subset B_{l}^{\prime\prime}$, then $Z_{l}$ contains at least one reduced irreducible component $B$ of $B_{l}^{\prime\prime}$.
On the other hand, by construction of $\os_l$ the scheme $Z_{l}$ does not contain any such $B$,
since $Z_{l}\cap P=\emptyset$ and $P\cap B\neq\emptyset$.

Thus for $l=n$ we have sections $\os_i$, $i=1,\dots n$, satisfying 
$\os_i\big|_{\PP^{n-1}}=t_i^{d_i^\prime}$, $\os\big|_{\dub{Z(s)} }= (s t_\infty^{d_i^\prime-d_i})\big|_{\dub{Z(s)} }$, %
and such that 
$Z_{red}(\os_1,\dots \os_{n-1})\not\subset B_{n-1}$.
To finish the proof it is enough %
to note that since $k$ is perfect, it follows  that $Z_{red}( \os_1,\dots \os_{n-1})-  \Sing Z_{red}( \os_1,\dots \os_{n-1})$ is open dense in $Z_{red}( \os_1,\dots \os_{n-1})$. 
Hence  $Z_{red}(\os_1,\dots \os_{n-1})\not\subset B_{n-1}$ implies $Z_{red}(\os_1,\dots \os_{n-1})\not\subset B_{n}$.
\end{proof}

\begin{lemma}\label{lm:iOScormove}
Let $c= (Z\subset A^n,v\colon \mathcal V\to \A^n,\phi,g\colon \mathcal V\to Y)\in Fr_n(\pt_k,Y)$ for some open $Y\subset \A^e$.
Then for any $i$, $0\leq i\leq n$, there exist $c^+,c^-\in Fr_n(\pt_k,Y)$,
such that $c^+$ and $c^-$ are $(i)$-simple correspondences, and $[c]=[c^+]-[c^-]\in \overline{\mathbb ZFr_n}(\pt_k,Y)$.
\end{lemma}
\begin{proof}
The proof follows immediately from Lemma \ref{lm:IndBase} and Lemma \ref{lm:IndStep}. 
\end{proof}

\begin{lemma}\label{lm:IndBase}
Let $Y\subset \A^e_k$ be an open subscheme, and $c= (Z\subset A^n,v\colon \mathcal V\to \A^n,\phi,g\colon \mathcal V\to Y)\in Fr_n(\pt_k,Y)$.
Then  there exist $c^+,c^-\in Fr_n(\pt_k,Y)$, such that $c^+$ and $c^-$ are $(n)$-simple correspondences, and $[c]=[c^+]-[c^-]\in \overline{\mathbb ZFr_n}(\pt_k,Y)$.
\end{lemma}
\begin{proof}
By Serre's theorem \cite[Theorem 5.2]{Hartshorne-AlG} we can choose integers $d_i$ and sections $s_i\in \Gamma(\PP^n,\cO(d_i))$,$i =1\dots  n$, $s_i/t^{d_i}_\infty=\phi_i\big|_{Z(I(Z)^2)}$, $s_i\big|_{\PP^{n-1}}=t_\infty^{d_i}$, where $\P^{n-1}\subset \P^n$ is the subspace at infinity and $t_\infty\in \cO(1)$, $Z(t_\infty)=\P^{n-1}$.
Similarly we can choose sections $e_i\in \Gamma(\PP^n,\cO(l_i))$, $1\leq i\leq k$,
$e_i/t^{l_i}_\infty\big|_{Z(I(Z)^2)}=g_i\big|_{Z(I(Z)^2)}$,
where the $g_i$'s are the coordinates of the composition $\mathcal V\xrightarrow{g} Y \hookrightarrow\A^e$.
The functions $\lambda v^*(s_i/t_\infty^{d_i})+(1-\lambda)(\varphi_i)$ and $\lambda v^*(e_i/t_\infty^{l_i})+(1-\lambda)g_i$
gives a homotopy from $c$ to the framed correspondence $$(\A^n-(Z(s_1,\dots s_n)-Z), Z, (s_i/t_\infty^{d_i}), (e_i/t_\infty^{l_i})).$$
Then applying Lemma \ref{prelm:GSO} we can change $c$ in such way that for a new vector of sections $(s_i)$ we have 
\begin{equation}\label{eq:baseSCSing} 
Z_{red}(s_1,\dots s_{n-1})\not\subset \bigcup\limits_{1\leq j < n} \Sing Z_{red}(s_1,\dots s_j).
\end{equation}
Since the original framed correspondence $c$ is equivalent, up to homotopy, to the resulting framed correspondence, we denote the resulting framed correspondence by the same symbol $c\in Fr_n(\pt_k,Y)$. %

Consider the closed subscheme of dimension one
$\hat Z_n=Z_{red}(s_1,\dots s_{n-1}) \subset \mathbb P^n$.
Since $k$ is perfect, the generic points of $\hat Z_n$ are smooth. 
Let $C$ be the union of the irreducible components of $\hat Z_n$ that intersect $Z$, where $Z$ is the support of $c$.
Denote $\overline s_n=s_n\big|_{C}$.
Since $Z_{red}(\overline s_n) = Z_{red}(s_1,\dots s_n)=Z\amalg Z^\prime$ there are line bundles $\cL$ and $\cL^\prime$ on $C$, %
such that $\cL\otimes\cL^\prime=\cO(d_i)$,
with sections $\overline s \in \Gamma(C,\cL)$, $\overline s^\prime\in \Gamma(C,\cL^\prime)$,
such that $Z(\overline s)=Z$, $Z(\overline s^\prime)\cap Z=\emptyset$, and $\overline s_n = \overline s\cdot \overline{s}^\prime$.

Denote 
$$D = C\cap \left( \overline{\hat Z_n - C} \cup \bigcup\limits_{1\leq j<n} \Sing(Z_{red}(s_1,\dots s_j))\cup Z^\prime \cup U^c  \right),$$
where
$U^c$ is the complement in $\PP^n$ of the open subscheme $U = g^{-1}(Y)$,
and $g = (e_j/t_\infty^{l_j})_j\colon \A^n\to \A^e$.
It follows from \eqref{eq:baseSCSing} that $D$ is proper in $C$. Note that by definition $C-D$ is smooth.
Denote $D_1 = D \cap \PP^{n-1}$, $D_2=D-D_1$. %
Now applying Lemma \ref{lm:onCurveMove} to the curve $C$,
the closed subsets $D_1$, $D_2$, and $B=D$, the section $t_\infty$ of the ample sheaf $\mathcal O(1)$, and the invertible sheaf $\mathcal L$,
for all 
$d_n^\prime>N$, for some $N\in\mathbb Z$,
for all field extensions $K/k$, $\deg K>R(d)$, for some $R(d)\in\mathbb Z$,
we find a sections $s_n^-\in \Gamma(C,\mathcal L(d_n^\prime))$, $s_n^-\in \Gamma(C,\mathcal O(d_n^\prime))$
such that 
$Z(s^+)$ and $Z(s^-)$ are smooth, 
$s^+\big|_{D_K}=(s\cdot s^-)\big|_{D_K}$, $s^-\big|_{Z(s_1,\dots s_n)\cup D_2}= t_\infty^{d_n^\prime}$, 
and $s^+$, $s^-$ are invertible on 
$D_K=D\times \Spec K$.

Define the correspondences $c^+_K$ and $c^-_K$ in $Fr_n(\pt_K,Y\times\Spec K)$ as 
$$\begin{aligned}
c^+ &=& ( Z(s_1,\dots s_{n-1},s^+) ,& \A^n- Z^\prime ,& (s_1/t_\infty^{d_1},\dots s_{n-1}/t_\infty^{d_{n-1}} , f^+) ,& g ),\\
c^-& =& (Z(s_1,\dots s_{n-1}, s^-),& \A^n- Z^\prime,& (s_1/t_\infty^{d_1},\dots s_{n-1}/t_\infty^{d_{n-1}} , f^-),& g ).
\end{aligned}$$%
where 
$f^+\in \cO(\A^n)$ is a lift of a regular function $\overline{s}^\prime\cdot s^+/ t_\infty^{d^\prime_n+d_n}\in \cO(C-(C\cap \PP^{n-1}))$,
and $f^-\in \cO(\A^n)$ is a lift of $\overline{s}^\prime \cdot \overline{s}\cdot s^-/ t_\infty^{d^\prime_n+d_n}\in \cO(C-(C\cap \PP^{n-1}))$.

Thus we see that $c^+_K$ and $c^-_K$ are $(n)$-simple framed correspondences and $[c_K] = [c^+_K]-[c^-_K]\in \overline{\mathbb Z Fr_n}(\pt_K,Y\times\Spec K)$, where $c_K$ is the image of $c$ by base change from $k$ to $K/k$.
Now using the finite descent from Section \ref{sect:Finite Descent} we get that $[c]=[c^+]-[c^-]\in \overline{\mathbb ZFr_n}(\pt_k,Y)$ for some $(n)$-simple framed correspondences $c^+$ and $c^-$. 
More precisely, we consider a pair of extensions such that $(\deg K_1/k,\deg K_2/k)=(\deg K_1,\chark k)=(\deg K_2,\chark k)$.
Then in the notation of lemma \ref{cor:FinDesN} we can define
$c^+ = pr\circ (c^+_{K_1} \oplus c^+_{K_2})\circ  L$,
and similarly for $c^-$.
\end{proof}

\begin{lemma}\label{lm:IndStep}
Let $i=1\dots n-1$.
Then for any $(i+1)$-simple framed correspondence 
$c= (Z\subset A^n,v\colon \mathcal V\to \A^n,\phi,g\colon \mathcal V\to Y)\in Fr_n(\pt_k,Y)$ for some open $Y\subset \A^e$, 
there exist $c^+,c^-\in Fr_n(\pt_k,Y)$, such that $c^+$ and $c^-$ are $(i)$-simple correspondences, and $[c]=[c^+]-[c^-]\in \overline{\mathbb ZFr_n}(\pt_k,Y)$.
\end{lemma}
\begin{proof}
Consider the closed subscheme of dimension one
$\hat Z_i=Z_{red}(s_1,\dots s_{i-1})\cap Z(s_{i+1},\dots s_n)$ in $\mathbb P^n$.
We will prove that 
$\hat Z_i\neq \Sing \hat Z_i$.

Since $c$ is $(i)$-simple $Z\subset Z_{red}(s_1,\dots ,s_i) - Sing(Z_{red}(s_1,\dots ,s_i))$,
and so any closed point $z\in Z$ is a smooth point of $Z_{red}(s_1,\dots ,s_i)$. 
Hence $\dim T=n-i$, where $T=T_z(Z_{red}(s_1,\dots s_i))$ is the tangent vector space. 
On the other hand, by assumption $Z^\prime$ is smooth, where $Z^\prime=Z_{red}(s_1,\dots s_i)\cap Z(s_{i+1},\dots s_n)$. 
Hence the gradients $d \phi_j$ of the functions
$\phi_j=s_j/t_j^{d_j}$, $j> i$,
are linearly independent on $T$;
and consequently the gradients $d\phi_j$, $j>i$, are linearly independent on the tangent space $T_z(Z_{red}(s_1,\dots s_{i-1}))\supset T$.
Thus $z$ is a smooth point on $\hat Z_i$, and hence there is a smooth Zariski neighbourhood of $z$ in $\hat Z_i$. %

The rest of the proof is similar to the proof of Lemma \ref{lm:IndBase}. 
Namely, let $C$ be the union of the irreducible components of $\hat Z_i$ that intersect $Z_i$. 
Then there are line bundles $\cL$ and $\cL^\prime$ on $C$,
such that $\cL\otimes\cL^\prime=\cO(d_i)$,
and sections $\overline s \in \Gamma(C,\cL)$, $\overline s^\prime\in \Gamma(C,\cL^\prime)$,
such that $Z(\overline s)=Z_i$, $Z(\overline s^\prime)\subset Z^\prime$, where $Z\amalg Z^\prime = Z(s_1\dots s_n)$.
Denote 
$$D = C\cap ( \overline{\hat Z_i - C} \cup \bigcup\limits_{j<i} Sing(Z_{red}(s_1,\dots s_j))\cup Z^\prime \cup U^c  ),
D_1 = D \cap \PP^{n-1}, D_2=D-D_1.
$$
Since $c$ is $(i)$-simple, $Z\subset \PP^n-\bigcup_{j<i} Sing(Z_{red}(s_1,\dots s_j))$. Hence $D\subset \hat Z_i$ is proper. By the above $B=D\cup \Sing C$ is proper too.
Applying Lemma \ref{lm:onCurveMove} to the curve $C$, 
and the proper closed subsets $D_1$, $D_2$, and $B=D\cup \Sing C$ %
we obtain the claim
over all field extensions $K/k$, $\deg K>R$ for some integer $R$.
The finite descent of Section \ref{sect:Finite Descent} finishes the proof as in the previous lemma. 
\end{proof}

\begin{lemma}\label{prelm:SmCseS}
Let $C$ be a reduced projective curve over a 
perfect field $k$, and $D\subset C$ a proper closed subscheme such that $C-D$ is smooth.
Let $\mathcal L$ be an invertible sheaf on $C$,
$\cO(1)$ an ample invertible sheaf on $C$,
and $r$ and $r^\prime$ invertible sections of $\mathcal L$ and $\cO$ on $D$.
Then there exists an integer $N$ such that for all $d > N$ there exists an integer $R(d)$ such that for all field extensions $K/k$ of degree $\deg_k K>R(d)$, there exists a section $s\in \Gamma(C,\mathcal L(d))$ such that $Z(s)$ is smooth and $s\big|_D=r\cdot {r^\prime}^d$.
\end{lemma}
\begin{proof}

Consider the affine space $\Gamma_d \subset \Gamma(C,\cL(d))$, $s\in \Gamma_d$ iff $s\big|_{D}=r\cdot {r^\prime}^d$.
Consider the universal section $\tilde s\in \Gamma(C\times\Gamma_d,\cL(d))$, and
its vanishing locus $Z(\tilde s)$, %
which is a closed subscheme $Z(\tilde s)\subset C\times\Gamma_d$.
The image of $\Supp\Omega_{Z(\tilde s)/\Gamma_d}\subset Z(\tilde s)$ under the projection to $\Gamma_d$ is the closed subscheme which parametrizes the set of sections $s$ such that $Z(s)\subset C$ is non-reduced. 
Denote this image by $B_d$ and let $U_d\subset \Gamma_d$ be the complement of $B_d$.

To find a section $s$ satisfying the requirements of the lemma is equivalent to finding a rational point in $U_d$ for all $d$ greater than some integer $N$.
We want to prove that there exists $R$ such that for all $K$ of degree $\deg K/k>R$ there exists $s\in \Gamma_d(K)$. %
Since $U_d$ is an open subscheme in an affine space over $k$, it suffices to show that there exists some integer $N$ such that for all $d>N$ we have $U_d\neq\emptyset$.
At the same time, to prove that $U_d\neq\emptyset$ it suffices to prove this over an algebraic closure $\overline k/k$,
that is,
$(U_d)_{\overline k}\neq\emptyset$. %

Thus we need to prove that $U_d\neq \emptyset$ for all $ d$ greater than some $N$
under the assumption that $k$ is algebraically closed.
For an algebraically closed field $k$ the property that $Z(s)$ is non-reduced for some $s\in \Gamma_d$ means that there is some point $p\in C$, such that $s\big|_{Z(I(p)^2)}=0$.
Since the sections $r$ and $r^\prime$ are invertible, we can assume in addition that
$p\in \osuC$, where $\osuC=C-D$.
Consider the closed subscheme $E_d\subset \Gamma_d\times \osuC$, $E=\{(s,p)| s\big|_{Z(I(p)^2)}=0 \}$.
Then $B_d$ is the image of $E_d$ under the projection to $\Gamma_d$.

We claim that there exists some $d\in \mathbb Z$ such that for each point $p\in \osuC$
we have
\begin{equation}
\label{eq:codim-claim}
\codim_{\Gamma_d}(\{s\in \Gamma_d | s\big|_{Z(I(p)^2)}=0\}) = 2.
\end{equation}
Then \eqref{eq:codim-claim} implies the lemma. 
Indeed,
if for all $p\in \osuC$, $\codim_{\Gamma_d}(s\in \Gamma_d | s\big|_{Z(I(p)^2)}=0) = 2$, then
$\dim(E_d)\leq \dim (\osuC) + \dim(\Gamma_d)-2$, and hence
$\codim_{\Gamma_d}(B_d) = \dim({\Gamma_d}) - \dim(B_d)\geq \dim_{\Gamma_d} - \dim(E_d) \geq 1$.

To prove \eqref{eq:codim-claim} it suffices to prove that for some $d\in \mathbb Z$ and for all $p\in \osuC$,
the restriction homomorphism $r_{p}^{d}\colon \Gamma(C,\cL(d)) \to \Gamma( Z(I(p)^2)\amalg D,\cL(d) )$ is surjective.
Consider the scheme $\osuC\times C$ as a relative curve over $\osuC$, 
and let $\Delta\subset \osuC\times C$ be the graph of the embedding $\osuC\hookrightarrow C$.
Then the set of points $p\in \osuC$ such that $r_p^d$ is not surjective is equal to 
$$W_d = \Supp  pr_*(\Coker(\cL(d)\to j^*j_*\cL(d)) )  ,$$ %
where $pr\colon \osuC\times C\to \osuC$, $i\colon \osuC\times D\to \osuC\times C$, $j\colon Z\to \osuC\times C$, $Z=Z(I(\Delta)^2)\cup \osuC\times D$, and $j_*$ and $j^*$ are the direct and inverse images of coherent sheaves.
Since $\cO(1)$ is ample, it follows that for each $p\in \osuC$ there is $N$ such that for all $d>N$ the restriction homomorphism $r^d_p$ 
is surjective. %
So \eqref{eq:codim-claim} follows since $\osuC$ is a noetherian scheme of finite Krull dimension. %
\end{proof}

\begin{lemma}\label{lm:onCurveMove}
Let $C$ be a reduced projective curve over a perfect field $k$,
$D_1\amalg D_2=D\subset B\subset C$ be closed subsets such that $C-B$ is smooth and non-empty;
let $s\in \Gamma(C,\cL)$ be a section in some invertible sheaf $\cL$ on $C$ such that $s\big|_{D}$ is invertible; let $\cO(1)$ be any ample bundle on $C$ with a section $t\in \Gamma(C,\cO(1))$ such that $Z_{red}(t)\subset D_1$.

Then for all $d>N$, for some $N\in \mathbb Z$, for all field extensions $K/k, \deg_k K>R(d)$, for some $R(d)$, there exist $s^+\in \Gamma(C_K,\cL(d))$, $s^-\in \Gamma(C_K,\cO(d))$ such that  $Z(s^+)$ and $Z(s^-)$ are smooth, $s^+\big|_{D_K}=(s\cdot s^-)\big|_{D_K}$, $s^-\big|_{Z(s)\cup D_2}= t$, and $s^+$, $s^-$ are invertible on $B_K=B\times \Spec K$ (and consequently on $D_K=D\times \Spec K$).

\end{lemma}
\begin{proof}
Since $D_1\cap D_2=D_1\cap Z(s) = D_2\cap Z(s)=\emptyset$, and since $B$ is a zero-dimensional scheme, it follows that $B$ splits into a disjoint union of
\begin{align*}
B_4 = B- (D\cup Z(s) ), B_1 = B - (D_2\cup Z(s)\cup B_4), \\
B_2 = B- (D_1\cup Z(s)\cup B_4), B_3 = B-(D\cup B_4).
\end{align*}
Let $r_1$ %
be any invertible sections of $\cO(1)$ on $B_1$, and let $w$ denote any invertible section on $\cL$ on $B_3\cup B_4$. %

Applying Lemma \ref{prelm:SmCseS} to the closed subset $B$, the line bundle $\cO(1)$, and the invertible section $r_1\oplus t^d\big|_{B_2\cup B_3\cup B_4}$, we see that
there exists $N_1$, such that for all $d>N_1$ there exists $R_1(d)$ such that for all $K/k, \deg K/k>R_1(d)$,
there exists a section $s^-\in \Gamma(C_K,\cO(d))$ such that %
$s^-\big|_{Z(s)\cup D_2}=t^d$, $s^-\big|_{D_1}=r_1^d$, and such that $Z(s^-)$ is smooth.

Applying Lemma \ref{prelm:SmCseS} to the closed subset $B$, the line bundles $\cL(N_1)$, $\cO(1)$ and the invertible section
\begin{align*}
  & (s \cdot s^-\big|_{D}) \amalg (w \cdot t^{N_1}) \in \Gamma((B_1\cup B_2)\amalg (B_3\cup B_4),\cL(N_1)), \\
  & r_1\amalg t\in \Gamma(B_1\amalg (B_2\cup B_3\cup B_4),\cO(d)),
\end{align*}
we see that 
there exists $N$ such that for all $d>N$ %
there exists $R(d)$ such that for all $K, \deg K/k> R(d)$, 
there exists a section $s^+\in \Gamma(C_K,\cL(d))$ such that 
$s^+\big|_{D_2} = s t^{d}\big|_{D_2}$, $s^+\big|_{D_1}=s r_1^d\big|_{D_1}$, $s^+\big|_{B_3\amalg B_4}=w t^d$,
and $Z(s^+)$ is smooth. %

So we get the sections $s^+$, $s^-$ with the required properties $s^+ = s s^-\big|_{D}$ and $Z(s)\subset C-B$, $Z(s)$ is smooth. 
\end{proof}

\begin{lemma}\label{lm:standFrptpt}
For any framed correspondence $c\in Fr_n(\pt_k,\pt_k)$ there is a pair of standard framed correspondences $c^+,c^-\in Fr_1(\pt_k,\pt_k)$ such that $[c]=[c^+]-[c^-]\in \overline{\mathbb ZF}(\pt_k,\pt_k)=H_0(\mathbb ZF(\Delta^\bullet,\pt_k))$.
\end{lemma}
\begin{proof}
Following the original strategy of the proof of \cite[Lemma 5.4]{Nesh-FrKMW}
we see that it suffices to consider the case of \begin{equation}\label{eq:hypcor}c = (Z(f),\A^1_k, f,pr)\in Fr_1(\pt_k,\pt_k),\end{equation} where $f\in k[\A^1_k]$, $pr\colon Z(f)\to \pt_k$ is the projection. %
For completeness we recall the arguments from \cite[Lemma 5.4]{Nesh-FrKMW}.

Firstly, by Lemma \ref{lm:iOScormove} and \cite[Lemma 4.10]{Nesh-FrKMW} we can reduce to consider a simple correspondence $c\in Fr_1(\pt_k,\pt_k)$.
Now, let $c = (Z(f),\A^1_k, fg,pr)\in Fr_1(\pt_k,\pt_k)$, where $f\in k[\A^1_k]=k[t]$ $\deg f=n$, $g\in k[\A^1_k]=k[t]$, $Z(g)\cap Z(f)=\emptyset$, $pr\colon Z(f)\to \pt_k$ is the projection.
We prove that the class of $c$ is a sum of classes of correspondences of the form \eqref{eq:hypcor}.
Actually, let $g^\prime\in k[t]$ be a polynomial of degree $n-1$ such that $g\big|_{Z(f)}=g^\prime \big|_{Z(f)}$.
Then because of the homotopy given by $\lambda g^\prime +(1-\lambda)g$ we see that we can assume that $g^\prime=g$.
Then the class of the correspondence $c^\prime = (Z(g),\A^1_k, fg,pr)\in Fr_1(\pt_k,\pt_k)$ satisfies the induction assumption. 
Hence the claim follows, since $[c]=[(Z(fg),\A^1_k, fg,pr)]-[(Z(g),\A^1_k, fg,pr)]\in \mathbb ZFr_1(\pt_k,\pt_k)$. 

Now all what is needed is to show that the class of \eqref{eq:hypcor} is standard.
But this is true,
since the homotopy $\lambda x^n +(1-\lambda) f$ and Lemma \ref{sbl:Lambda} implies that $[c]=mh$, for $n=2m$, or $[c]=mh+1$, for $n=2m+1$.
\end{proof}
\begin{remark}
Alternatively we could say that the proof of the above lemma presented in \cite[Lemma 5.4]{Nesh-FrKMW} holds for an arbitrary infinite field and the deduce the finite field case using the finite descent form the section \ref{sect:Finite Descent}.
\end{remark}

\bibliographystyle{amsplain}

\begin{thebibliography}{XXXXX}


\bibitem{AGP-FrCanc}
A. Ananyevskiy, G. Garkusha, I. Panin,
Cancellation theorem for framed motives of algebraic varieties,
arXiv:1601.06642.

\bibitem{BM-KMW}J. Barge et F. Morel, Cohomologie des groupes lineaires, K-theorie de Milnor et groupes de Witt, C.R. Acad. Sci. Paris, t. 328, S ́erie I, p. 191-196, 1999.

\bibitem{AD}
A.~Druzhinin,
Strictly homotopy invariance of Nisnevich sheaves with GW-transfers,
arXiv:1709.05805.

\bibitem{DKMWhom} %
A.~Druzhinin, Steinberg relation in the motivic homotopy theory over a base, arXiv:1809.00087. %

\bibitem{Appendix-FD-MotPairs}
A.~E.~Druzhinin and J.~I.~Kylling,  Finite descent , the Appendix in 
 Motives of smooth affine pairs, preprint arXiv:1803.11388.

\bibitem{DrKold-ACorr}
A.~Druzhinin, H.~Kolderup,
Cohomological correspondence categories,
arXiv:1808.05803.

\bibitem{DP-Sur2}
A.~Druzhinin, I.~Panin, %
Surjectivity of the \'{e}tale excision map for homotopy invariant framed presheaves,
arXiv:1808.07765.

\bibitem{ElHoKhSoYa-MotDeloop} 
E.~Elmanto, M.~Hoyois, A.~A.~Khan, V.~Sosnilo, M.~Yakerson,
%
Motivic infinite loop spaces,
arXiv:1711.05248.

\bibitem{Fasel-ChWittRing} J. Fasel. The Chow-Witt ring. Doc. Math. 12 (2007), 275–312.
\bibitem{Fasel-GroupsdeCW} J. Fasel. Groupes de Chow-Witt. Mm. Soc. Math. Fr. (N.S.) No. 113 (2008).

\bibitem{GP_MFrAlgVar} 
G.~Garkusha, I.~Panin,
Framed motives of algebraic varieties (after V. Voevodsky),
preprint arXiv:1409.4372.

\bibitem{GP-HIVth}
G.~Garkusha, I.~Panin,
Homotopy invariant presheaves with framed transfers,
preprint arXiv:1504.00884.

\bibitem{GNP_FrMotiveRelSphere}
G.~Garkusha, A.~Neshitov, I.~Panin,
Framed motives of relative motivic spheres, preprint arXiv:1604.02732.

\bibitem{Hartshorne-AlG}
R.~Hartshorne, Algebraic geometry, Springer, Graduated texts in mathematics: 52 (edited by S.~Axler, F.W.~Gehring, P.R.~Halmos), New York-Heidelberg, 1997.

\bibitem{PHandIK-Steinberg}
Po Hu, Igor Kriz,
The Steinberg relation in A1-stable homotopy theory,
International Mathematics Research Notices, 2001, No. 17.

\bibitem{Nesh-FrKMW}
A.~Neshitov,
Framed correspondences and the Milnor-Witt $K$-theory,
Journal of the Institute of Mathematics of Jussieu, 
Vol. 17, Issue 4
(2018) , pp. 823-852.
%


%

\bibitem{M02}
F.~Morel, On the motivic $\pi_0$ of the sphere spectrum, in J.P.C. Greenlees (ed.), Axiomatic, Enriched and Motivic Homotopy Theory, 219-260, 2004 Kluwer Academic Publishers..

\bibitem{M-A1Top} 
F. Morel, A1-algebraic topology over a field, Lecture Notes in Mathematics, 2052. Springer,
Heidelberg, 2012.

\bibitem{Powell-Steinberg}
G.~M.~L.~Powell, Zariski excision and the Steinberg relation, (2002) preprint at 
http://math.univ-angers.fr/~powell/documents/2002/steinberg.pdf.

\bibitem{Voev-FrCor} V. Voevodsky, Notes on framed correspondences, unpublished, 2001. Also available at
math.ias.edu/vladimir/files/framed.pdf

\bibitem{VMW} C.~Mazza, V.~Voevodsky, C.~A.~Weibel,
Lecture Notes on Motivic Cohomology,
Clay mathematics monographs, ISSN 1539-6061 ; v. 2.



\end{thebibliography}

\end{document}